\documentclass[a4paper,12pt,oneside,reqno]{amsart}
\usepackage{hyperref}
\usepackage[headinclude,DIV13]{typearea}
\areaset{15.1cm}{25.0cm}
\usepackage{amsfonts,amssymb,amsmath,amsthm,bbm}
\usepackage{cancel} 
\usepackage[latin1] {inputenc}
\usepackage{graphicx, psfrag}

\DeclareMathOperator{\arcsinh}{arcsinh}

\newtheorem{theorem}{Theorem}[section]
\newtheorem{lemma}[theorem]{Lemma}
\newtheorem{proposition}[theorem]{Proposition}

\newtheorem{claim}[theorem]{Claim}

\newtheorem{example}{Example}

\theoremstyle{definition}
\newtheorem{definition}[theorem]{Definition}

\theoremstyle{remark}
\newtheorem*{remark}{Remark}

\def\paragraph#1{\noindent \textbf{#1}}

\numberwithin{equation}{section}

\def\<{\langle}
\def\>{\rangle}

\def\e{\epsilon}

\def\R{{\Bbb R}}  
\def\N{{\Bbb N}}  
\def\Z{{\Bbb Z}}  
\def\Q{{\Bbb Q}}  
\def\E{{\Bbb E}}  

\def\H{{\Bbb H}}

\let\cal=\mathcal

 \def \e {{\varepsilon}}

 \def \ba {\begin{array}}
 \def \ea {\end{array}}

 \newcommand{\be}{\begin{equation}}
 \newcommand{\ee}{\end{equation}}

\newcommand{\bea}{\begin{eqnarray}}
 \newcommand{\eea}{\end{eqnarray}}
\def\TH(#1){\label{#1}}\def\thv(#1){\ref{#1}}
\def\Eq(#1){\label{#1}}\def\eqv(#1){(\ref{#1})}

 \def \1{\mathbbm{1}}

\begin{document}

 \title[Aperiodic Sequences and Aperiodic Geodesics]
{Aperiodic Sequences and Aperiodic Geodesics}
\author[V. Schroeder]{Viktor Schroeder}
 \address{V. Schroeder\\ Institut f\"ur Mathematik der Universit\"at Z\"urich\\}
\email{viktor.schroeder@math.uzh.ch}
\author[S. Weil]{Steffen Weil}
 \address{S. Weil\\ Institut f\"ur Mathematik der Universit\"at Z\"urich\\}
\email{steffen.weil@math.uzh.ch}

\date{\today}

 \begin{abstract} 
We introduce a quantitative condition on orbits of dynamical systems which measures their aperiodicity.
We show the existence of  sequences in the Bernoulli-shift and geodesics on closed hyperbolic manifolds which are as aperiodic as possible with respect to this condition.
  \end{abstract} 

\thanks{The authors acknowledge the support by the Swiss National Science Foundation (Grant: 135091).} 
\maketitle

\section{Main Results.} 

In this section we state  our main results in the case of sequences in a finite alphabet and of geodesics in hyperbolic manifolds.
Denote by $\N_0$ the natural numbers including $0$ and let $\N = \N\setminus \{ 0\}$.
Given a finite set $\cal{A}$ with $ k\geq 2$ elements, let $\Sigma = \cal{A}^{\Z}$ be the set of  
biinfinite sequences in the alphabet $\cal{A}$, which we call \emph{words}.
With $[w(i)\ldots w(i+l)]$ denote the subword of $w\in \Sigma$ starting at \emph{time} $i\in \Z$ and of \emph{length} $l\in \N_0$.
For a word $w\in\Sigma$ define the \emph{recurrence time} $R_{w}^i: \N_0 \to \N \cup \{\infty\}$ at time $i\in \Z$ by
\be	\nonumber
	R_w^i(l) = \min\{s \geq 1 : [w(i+s)\ldots w(i+s+l)]= [w(i)\ldots w(i+l)]\},
\ee
(i.e. the first instant when the sub word $[w(i)\ldots w(i+l)]$ of $w$ is seen again), and by
\be  \nonumber
	 R_w(l) := \min\{R_w^i(l) : i \in \Z\}.
\ee
For a periodic word $w\in \Sigma$ with period $p\in \N$, i.e. $w(i)=w(i+p)$ for all $i\in\Z$, we have
$R_w(l)\le p$ for all $l\in \N_0$. 
Thus, if $R_w$ is unbounded, then $w$ is aperiodic and we view the growth rate of
$R_w$ as a measure for the aperiodicity of the word $w$.
Note that $R_w$ is nondecreasing and by a trivial counting argument we have
$R_w(l) \le k^{l+1}$ for every word $w$, in particular
\be \nonumber
         \lim_{l\to \infty} \frac{1}{l} \ln R_w(l) \leq \ln(k).
\ee 
One of our main results is the existence of words $w$ such that the growth rate is as near as 
possible to this bound.

\begin{theorem}  \label{MT1}
Let $\varphi:\N_0\to [0, \infty)$ be a non-decreasing function such that
\be
\label{AsymptoticCondition}
	\lim_{l\to \infty} \frac{1}{l}\ln(\varphi(l)) \leq  \delta \ln(k)
\ee
for some $0<\delta < 1$.
Then there exist $l_0=l_0(\varphi, k,\delta) \in \N_0$ and a word $w\in \Sigma$ such that,
for every $l_0\leq l \in \N_0$, we have 
$R_w(l) \geq \varphi(l)$.
\end{theorem}

Now let $M$ be a closed $n$-dimensional hyperbolic manifold, where $n\geq 2$. 
Let $i_M>0$ denote the injectivity radius of $M$  and let  $d$ be the Riemannian distance function on $M$.
For a unit speed geodesic $\gamma:\R\to M$ we define the \emph{recurrence time} 
$R_{\gamma}^{t_0} : [0, \infty) \to [i_M/2, \infty]$ at time $t_0\in \R$ by 
\be 	\nonumber
	R_{\gamma}^{t_0} (l) = \inf \{ s > i_M/2 :  d(\gamma(t_0+  t), \gamma(t_0+s+t))< \frac{i_M}{2} \text{ for all } 0\leq t \leq l  \}.
\ee
and
\be	\nonumber
	 R_{\gamma}(l) := \inf \{ R_{\gamma}^{t_0} (l) : t_0\in \R\}.
\ee
If $\gamma$ is a periodic geodesic, then $R_{\gamma}$ is bounded and again one can view the growth rate of
$R_{\gamma}$ as a measure for the aperiodicity of $\gamma$.

\begin{theorem} \label{MT}
Let $\varphi: [0, \infty) \to [0, \infty)$ be a non-decreasing function such that
\be
\label{AsymptoticCondition2}
	\lim_{l\to \infty} \frac{1}{l} \ln(\varphi(l)) \leq \delta (n-1)
\ee
for some $0<\delta <1$. If  $i_M>2 \ln(2)$ then there exist $ l_0 = l_0(\varphi, \delta,n,i_M)\geq 0$ and a unit speed geodesic 
$\gamma: \R \to M$ such that for all $l\geq l_0$, we have $R_{\gamma}(l) \geq \varphi(l)$.
\end{theorem}

The theorems will be shown in greater generality.

\begin{remark} 
The bounds $\ln(k)$ and $n-1$ equal the topological entropies of the respective dynamical systems.
Moreover, we believe that the assumption on the injectivity radius in Theorem \ref{MT} is not necessary.
A version of this theorem is also true if $M$ is of strictly negative curvature.
However, for the sake of clarity of the paper we restrict to these assumptions.
\end{remark}

\emph{Organization of the paper.} In Section $2$ we will introduce the measure of aperiodcitiy  for general dynamical systems and deduce immediate properties.
In Section $3$ and $4$ we examine two examples and  state the main results, namely of the Bernoulli-shift and the geodesic flow on a closed hyperbolic manifold.
These will be proven in Section $5$.
\\

\emph{Acknowledgement.} 
We want to thank Shahar Mozes for helpful discussions.
The second author would like to thank Jean-Claude Picaud for many fruitful discussions and comments and the University of Tours for its hospitality during his stay in January 2012.

\section{$F$-Aperiodic Points.}
Let $(X,d)$ be a compact metric space and let $T : X \to X$ be a given continuous transformation. 
 For $n\in\N_0$ let $T^n$ be the $n$-times  composition of $T$ (where $T^0=id_X$) and 
for a point $x\in X$ let $T^nx$ be the point in the orbit $\cal{T}(x) := \{T^nx\}_{n\in \N_0}$ of $x$ at \emph{time} $n$.
Let moreover $\mu$ be a finite Borel-measure on the 
Borel-$\sigma$-algebra $\cal{B}$ of $(X,d)$ such that $T$ is measure-preserving; see \cite{Einsiedler}.

A point $x\in X$ is called \emph{periodic} (with respect to $T$) if 
there exists an integer $p\in \N$, called a \emph{period} of $x$, 
such that $T^px=x$. 
Denote by $\cal{P}_T$ the $T$-invariant set of $T$-periodic points of $X$.
A point is called \emph{aperiodic}, if it is not periodic.

A point $x \in X$ is \emph{recurrent} with respect to $T$, if for any 
$\varepsilon>0$ there exists $s=s(x,\varepsilon) \in \N$ such that $d(T^sx,x) < \varepsilon$. 
Periodic points are obviously recurrent.
The set $\cal{R}_T$ of recurrent points is nonempty (see \cite{Furstenberg}) and $T$-invariant.
However $s(T^i x,\varepsilon)$ can differ from $s(x,\varepsilon)$ in general, 
unless $T$ is an isometry on its orbit $\cal{T}(x)$;  that is, $d(T^{i+s}x, T^ix) = d(T^sx,x)$ for all $i$ and $s\in\N_0$. 
We recall that by the Poincar\'e-recurrence theorem, $\mu$-almost every point is recurrent.

In this paper
we give a quantitative version of recurrence and aperiodicity.
Given a point $x\in X$ and a \emph{time} $i\in \N_0$, 
we ask for a
lower bound on the \emph{shift} $s$ such that $T^{i+s} x$ is allowed to be $\varepsilon$-close to $T^ix$:

\begin{definition} \label{Def1}
For a non-increasing function $F : (0, \infty) \to [0, \infty)$ a point $x \in X$ is 
called \emph{$F$-aperiodic} at  \emph{time} $i\in\N_0$ if for every  $\varepsilon>0$, whenever
\be \nonumber
	d(T^ix, T^{i+s}x)<\varepsilon 	
\ee
for some $s \in \N$, then $s > F(\varepsilon)$.
If $x$ is $F$-aperiodic at every time $i \in \N_0$ then it is called \emph{$F$-aperiodic}. 
\end{definition}

We emphasize that although we called the condition "$F$-aperiodic", 
a periodic point $x$ is 
$F$-aperiodic for a suitable bounded function $F$. 
However, if the function $F$ is unbounded, an $F$-aperiodic point must be aperiodic.
Moreover, if $x$ is not recurrent, then it is easy to find an 
unbounded function $F$ such that $x$ is $F$-aperiodic at least  at time $0$.

Let $F : (0, \infty) \to [0, \infty)$ be a given non-increasing function.
Clearly, if a non-increasing function $\bar F$ satisfies $\bar F(s)\leq F(s)$ for all 
$s\in (0, \infty)$ then an $F$-aperiodic point is also $\bar F$-aperiodic.
On the other hand, using the upper box dimension $\dim_B(X)$ for metric spaces,  
we obtain an  upper bound on the  growth rate (as $\e$ tends to $0$) of functions $F$ such that an $F$-aperiodic point might exist.
For $\e>0$ let $N(X,\e)$ denote the largest number of disjoint metric balls of radius $\e$. 
Then the upper box dimension (\cite{Robinson}) is given by
\be \nonumber
	\dim_B(X) = \limsup_{\e\to 0} \frac{\ln(N(X, \e))}{-\ln(\e)}.
\ee

\begin{lemma}\label{UpperBound} 
Let $x$ be an $F$-aperiodic point.
Then $\limsup_{\e\to 0}\frac{\ln(F(\varepsilon))}{\ln(2/\e)} \leq \dim_B(X)$. 
\end{lemma}

\begin{proof} 
Let $\varepsilon>0$.
If $B(T^{s_1}x, \varepsilon/2)\cap B(T^{s_2}x, \varepsilon/2)\neq \emptyset$ for some $0\leq s_1 < s_2\leq F(\varepsilon)$, 
we have $d(T^{s_1}x, T^{s_2}x) < \varepsilon_0$ which is impossible since $s_2-s_1 \leq F(\varepsilon_0)$.
Therefore the metric balls $B(T^sx, \varepsilon/2)$ must be disjoint for $s\leq F(\varepsilon)$.
Hence we have $F(\e) \leq N(X, \e/2)$.
\end{proof}

Moreover, since $F$ is independent of the time $i\in \N_0$, the set $\cal{F}_T \subset X$ of $F$-aperiodic points  is $T$-invariant. In the case when $(X, \cal{B}, \mu, T)$ is ergodic, $\cal{F}_T$ is either of full or of zero $\mu$-measure. 
When $\cal{P}_T$ is nonempty, this question is related to the distribution of periodic orbits.
In fact, let $x_0 \in \cal{P}_T$ be of minimal period $p_0$ and assume that $F(\e)\geq p_0$ for some $\e_{p_0}>0$. 
In the case when $F$ is continuous, we may choose $\e_{p_0} := \sup\{\e>0 : F(\e) \geq p_0\}$. 
Define  the \emph{critical neighborhood} of $x_0$ with respect to $F$ and $p_0$ by
\be \label{CriticalNbhd}
	\cal{N}_{x_0} := B(x_0, \e_{p_0}/2) \cap T^{-p_0}(B(x_0, \e_{p_0}/2)).
\ee
Whenever $x\in \cal{N}_{x_0}$ we have by the triangle inequality that $d(x,T^{p_0}x) <\e_{p_0}$, but $p_0\leq F(\e_{p_0})$. Thus, no point in $\cal{N}_{x_0}$ can be $F$-aperiodic and we see that the orbit of an $F$-aperiodic point must avoid the critical neighborhoods of periodic points.
If in addition $\mu(\cal{N}_{x_0})>0$ then the set of $F$-aperiodic points cannot be of full and must therefore be of zero $\mu$-measure.
Thus, we showed the following criterion. 

\begin{lemma}\label{Measure}
Assume $\cal{P}_T\neq \emptyset$ and let $x_0$ be a periodic point of period $p_0$ and $F(\e)\geq p_0$ for some $\e>0$.
If $\mu$ is ergodic and  positive on $\cal{N}_{x_0}$  then the set $\cal{F}_T$ has $\mu$-measure $0$. 
\end{lemma}

In particular, this result is interesting for the \emph{systolic point}   $x_0\in \cal{P}_T$ of  \emph{systolic period}  $p_0 \in \N$, that is, 
$x_0$ has minimal period $p_0$  and for every periodic point in $X$ of period $p$ we have $p\geq p_0$. 

\begin{lemma} \label{ClosedCond}
$F$-aperiodicity is a closed condition. 
\end{lemma}

\begin{proof}
Let $\{x_n\}_{n \in \N}$ be a sequence of $F$-aperiodic points in $X$ converging to $x\in X$. Let $i$ and $s \in \N$ be fixed. 
For $\varepsilon>0$ such that $d(T^ix, T^{i+s}x)< \varepsilon$  let $d:= \frac{1}{2}(\varepsilon - d(T^ix, T^{i+s}x))$. 
Since $T$ is continuous, there exists $N=N(i,s, d) \in \N_0$ such that for all $n\geq N$ we have $d(T^ix, T^ix_n)< d$ and $d(T^{i+s}x, T^{i+s}x_n)<d$.
From the triangle inequality we obtain 
\be \nonumber
	d(T^{i}x_n, T^{i+s}x_n) \leq d(T^ix_n, T^ix) + d(T^ix, T^{i+s}x) + d(T^{i+s}x, T^{i+s}x_n) < \varepsilon
\ee
for $n\geq N$ so that $s>F(\varepsilon)$ since $x_n$ is $F$-aperiodic. Hence, $x$ is also $F$-aperiodic.
\end{proof}

Finally, note that if $T$ acts as an isometry on the orbit $\cal{T}(x)$ of a point $x\in X$, then $x$ is $F$-aperiodic  as soon as it is $F$-aperiodic at a given time. For instance, we consider the rotation on the circle as a motivating example:

\begin{example}\label{Ex1}
\emph{
Let $\Z$ act on $\R$ by translations and let $X=\R / \Z$ be the compact quotient space with the induced metric $d$ obtained from the Euclidean metric. Given an irrational number $0<\alpha\in \R\setminus \Q$, we let $T=T_{\alpha}:X \to X$ be the automorphism induced by the translation $\tilde T : \R \to \R$, $\tilde T(x) := x+\alpha$. For $c>0$ we let $F_c: (0,\infty) \to [0,\infty)$, $F_c(t) = c t^{-1}$. In fact, since $\dim_B(X)=1$, $-1$ is the optimal exponent due to Lemma \ref{UpperBound}. The point $[0]$ is $F_c$-aperiodic if and only if every point $[x]$ is $F_c$-aperiodic and hence $\cal{F}_T$ is either empty or $X$ itself.
Moreover,
since $T$ is an isometry, $[0]$ is $F_c$-aperiodic as soon as it is $F_c$-aperiodic at time $0$.
The question for which $c$ and $\alpha$ there exist $F_c$-aperiodic points can be answered by classical Diophantine approximation; see for instance \cite{Dodson} for the following well-known results:
Let $\mu$ be the Lebesgue measure on $\R$.
For $\mu$-almost every $\alpha \in \R\setminus \Q$ we have $c_0(\alpha)=0$, where
\be \nonumber
	c_0(\alpha) = \inf\{c>0 : \text{there exist infinitely many } p\in \Z, q\in \N \text{ such that } \lvert \alpha - \frac{p}{q} \rvert< \frac{c}{q^2} \}.
\ee
However, there exists a set of Hausdorff-dimension one such that $c_0(\alpha)$ is positive. Such an $\alpha$ is called badly approximable.
The supremum $\sup_{\alpha\in \R\setminus \Q} c_0(\alpha)$ of this set, called the Hurwitz-constant, is equal to  $1/\sqrt 5$ and attained at the golden ratio.}

\emph{
First, let $\alpha$ such that $c_0(\alpha)=0$.
Then for $c>0$ we have for infinitely many $p\in \Z$, $q\in \N$,
\be \label{Diophantine}
	\lvert \tilde T^q 0 - p \rvert  = \lvert q \alpha - p \rvert = q \lvert \alpha - \frac{p}{q} \rvert < c q^{-1},
\ee
hence $q\leq  F_c( c q^{-1})$ and we see that $[0]$ is not $F_c$-aperiodic for any $c>0$. Thus, $\cal{F}_T$ is empty.
In particular, this shows that for $c>1/\sqrt 5$ the set $\cal{F}_{T}$ is empty for every $T=T_{\alpha}$, $\alpha \in \R\setminus \Q$  irrational.
However, for $\alpha$ a badly approximable number we have  $c_0(\alpha)>0$ and  for $c<c_0(\alpha)$ there are only finitely many $p$, $q$ as in \eqref{Diophantine}. Hence we can choose some $0<\bar c\leq c_0(\alpha)$ such that $[0]$ is $F_{\bar c}$-aperiodic and therefore $\cal{F}_T =X$.
\\
If we conversely assume that $[0]$ is $F_c$-aperiodic, then whenever $\lvert \tilde T^q 0 - p \rvert <\e$ for some $\e>0$ we have $q> F_c(\e) = c/\e > \frac{c}{q \lvert \alpha - p/q \rvert}$. Thus, $\lvert \alpha - \frac{p}{q} \vert > \frac{c}{q^2}$ for every $p\in \Z$, $q\in \N$ and $\alpha$ is necessarily  a badly approximable number.
}
\end{example}

In the following we are concerned with the examples of the Bernoulli-shift and the geodesic flow on a closed hyperbolic manifold where the question of existence of $F$-aperiodic points is more delicate.

\begin{remark}
A somewhat orthogonal problem has been studied by many authors. For instance, \cite{Bosh} showed that the rate of recurrence can be quantified in the case when $X$ has finite Hausdorff-dimension.
More precisely,  assume that  the $\alpha$-dimensional Hausdorff-measure $H_{\alpha}$ is $\sigma$-finite for some $\alpha>0$, then for $\mu$-almost every point $x\in X$ there exists a finite constant $c(x)\geq0$ such that
\be \nonumber
	\liminf_{n\to \infty} n^{1/\alpha}d(x, T^n(x)) \leq c(x).
\ee

Assume  that there exists a point $x\in X$ which is $F$-aperiodic at time $0$ for the function $F(\e)= c\cdot \e^{-\alpha}$ for some $c>0$ (compare with Lemma \ref{UpperBound}). 
Then it is not hard to show that for every $n>0$,
\be \nonumber
	n^{1/\alpha} d(x, T^nx) \geq c^{1/\alpha}.
\ee
The main point in our paper is that we study the recurrence for every point of the orbit and not only for the initial one. 
\end{remark}


\section{Sequences.}
Let  $\mathcal{A}$ be a finite set of $k\geq2$ elements which we call \emph{alphabet}. Let $\Sigma^+ = \{w : \N \to \mathcal{A}\}$ and  $\Sigma = \{w : \Z \to \mathcal{A}\} $  be the set  two-sided sequences in symbols from $\mathcal{A}$. 
The elements of $\Sigma$ are called \emph{words}. 
Given words $w$ and $\bar w$ in $\Sigma$ we let $a(w, \bar w) = \max\{ i \geq 0 : w(i) = \bar w(i)$ for $\lvert j \rvert \leq i \}$ for $w \neq \bar w$ and  define $\bar d(w, \bar w) := 2^{-a(w, \bar w)}$, and $\bar d(w, w) :=0$ otherwise.
Let $T$ denote the shift operator acting on $\Sigma$, with $T(w) = \bar w$ where $\bar w(i) = w(i+1)$.
Then, $(\Sigma, \bar d)$ is a compact metric space such that $T$ is a homeomorphism.
Moreover, let $\mathcal{B}$ denote the product $\sigma$-algebra of the power set $\mathcal{P}(\mathcal{A})$ of $\mathcal{A}$ which equals the Borel-$\sigma$-algebra of $(\Sigma,\bar d)$.  
Let (the probability measure) $\mu=\prod_{\Z}\mu_{\mathcal{A}}$ be the infinite product measure of $\cal{B}$ where $\mu_{{\mathcal A}}$ is a probability measure on $(\mathcal{A}, \mathcal{P}(\mathcal{A}))$. 
Then the \emph{Bernoulli-shift} $(\Sigma, \mathcal{B}, \mu, T)$ is ergodic. 
For details we refer to \cite{Einsiedler}.

Note that by definition of $\bar d$, two words are close if and only if the length of their subwords around position $0$ on which they agree is large. In particular, if $w\in \cal{R}_T$ then,  by recurrence applied to the word $T^iw$, for every length $l \in \N_0$ we can find an $s=s(i,l)\in\N$ such that $[w(i)\ldots w(i+l)]=[w(i+s)\ldots w(i+s+l)]$. 
In the case of sequences it is suitable to reformulate $F$-aperiodicity as follows (see Proposition \ref{EquiDef}).

\begin{definition} \label{AperiodicSequences}
For a non-decreasing function $\varphi : \N_0 \to [0, \infty)$  a  word $w \in \Sigma$ is called \emph{$\varphi$-aperiodic} at  \emph{time} $i\in\Z$, if for every  \emph{length} $l\in \N_0$, whenever
\be \label{Condition}
	 [w(i)\ldots w(i+l)] = [w(i+s)\ldots w(i+s+l)]	
\ee
for some \emph{shift} $s \in \N$, then $s > \varphi(l)$. If $w$ is $\varphi$-aperiodic at every time $i\in \Z$ it is called $\varphi$-aperiodic.
\end{definition}

A $\varphi$-aperiodic word $w \in\Sigma$ is $F$-aperiodic for the following function $F$.

\begin{proposition} \label{EquiDef}
A $\varphi$-aperiodic word $w \in \Sigma$  is $F$-aperiodic for $F(\e)=\varphi(- 2 \lceil \log_2(\e) \rceil )$. 
Conversely, an $F$-aperiodic word $w$ is $ \varphi$-aperiodic for $\varphi(l)=F(2^{-(l/2-1)})$. 
\end{proposition}

\begin{proof}
Let $i\in \Z$ and $s \in \N$.
For every $l\in \N_0$ such that $\bar d(T^{i}w, T^{i+s}w) \leq 2^{-l}$ we have $[w(i-l)\ldots w(i+l)]=[w(i-l+s)\ldots w(i+s+l)]$. Thus, for  $2^{-l}<\e \leq 2^{-(l-1)}$,
\be \nonumber
	 s> \varphi(2l)  = \varphi( - 2 \lceil \log_2(\e) \rceil ) = F(\e).
\ee 
Since $F(\bar \e) \leq F(\e)$ for $\bar \e \geq \e$, the first implication follows.

Conversely, if $w$ is $F$-aperiodic,  assume that $[w(i)\ldots w(i+l)]=[w(i+s)\ldots w(i+s+l)]$ for $s\in \N$, $l\in \N_0$ and let $\bar l= l/2$ if $l$ is even and $\bar l=(l-1)/2$ if $l$ is odd.
Hence, $\bar d(T^{i+\bar l}w, T^{i+\bar l +s}w) \leq 2^{-\bar l}$ and for every $2^{-\bar l}<\e\leq 2^{-(\bar l-1)}$ we have
\be \nonumber
	s>F(\e) \geq  F(2^{-(\bar l-1}) \geq  F(2^{-(l-3)/2}) =   \varphi(l).
\ee
This finishes the proof.
\end{proof}

If a $\varphi$-aperiodic  word contains a periodic subword of infinite length then the function $\varphi$ is bounded, whereas if a word is $\varphi$-aperiodic for an unbounded function, the word must be aperiodic.
We want to give some examples in order to make the definition more familiar, among them the prominent Morse-Thue-sequence:

\begin{example}
\emph{
First, let $a$, $b\in \cal{A}$. One checks that the (non-recurrent) words $w_1=\ldots bbbaaa\ldots $ and $w_2=..abaabaaabaaaab\ldots $  are $\varphi$-aperiodic  only for a function $\varphi$ such that $1=s>\varphi(l)$ for all $l\in \N_0$.
Both, the orbits of $w_1$ and $w_2$, come closer and closer to the periodic word $\ldots aaa\ldots $ with respect to the metric $\bar d$.
This is not the case for $\varphi$-aperiodic words when $\varphi$ is unbounded; see Proposition \ref{PeriodicDistance}.
}

\emph{
Consider the Morse-Thue recurrent sequence $w\in \{0,1\}^{\Z}$ which is determined as follows:
Let $a_0=0$, $b_0=1$. Then for $n\in \N_0$, let $a_{n+1}=a_nb_n$ and $b_{n+1}=b_na_n$ be finite words of length $2^{n+1}-1$.
Then $w$ is defined such that it satisfies  $[w(0)\ldots w(2^{n}-2)] = a_n$ and $[w(-n)]=[w(n-1)]$ for every $n\in \N$.
In particular, $w$ contains the sub words $a_{n+2} = a_nb_nb_na_n$. Hence for every length $l=2^n-1$, $w$ contains subwords of the form $WW$ where $W$ has length $l$. A function $\varphi$ such that $w$ is $\varphi$-aperiodic must therefore be bounded by $\varphi(2^n-1) \leq 2^n-1$ for every $n\in\N$.
On the other hand there are no sub words of the form $WWa$ where $a$ is the first letter of a sub word $W$ (see \cite{Morse}).
In other words, $w$ is overlap-free (which means that there are no sub words of the form $aWaWa$ for a finite sub word $W$ and a letter $a$), 
from which follows that there are even no sub words of the form $wWwWw$ for $w$ and $W$ finite subwords. Hence we may choose $\varphi(l) \geq l$.
We conclude that $w$ is at least $\varphi$-aperiodic for the function $\varphi(l)=l$, $l\in\N_0$.
}
\end{example}

The example shows that the set of $\varphi$-aperiodic words $\cal{F}_T= \cal{F}_T(\varphi)$ is nonempty for the unbounded function $\varphi (l) = l$ and moreover, the Morse-Thue sequence gives an explicit example of such a word.
However, let $a\in \cal{A}$ such that $\mu_{\cal{A}}(\{a\})>0$ and let $w=\ldots aaa\ldots $ be a  periodic word which is of systolic period $1$. Moreover, $\mu$ is positive on the critical neighborhood of $w$ and hence by Lemma \ref{Measure}, $\cal{F}_T$ is of zero $\mu$-measure unless $\varphi$ is strictly bounded by $1$.

Our main result for sequences is the following.  It will be proved in Section \ref{Paragraph2}.

\begin{theorem} \label{TheoremSequence}
Let  $\varphi : \N_0 \to [0, \infty)$  be a  non-decreasing unbounded function such that
there exists $c\in(1,k)$ satisfying
\be \label{B}
  k -  \lfloor \varphi(0) \rfloor - \sum_{l=1}^{\infty} \frac{ \lfloor \varphi(l) \rfloor - \lfloor \varphi(l-1) \rfloor}{c^l} \geq c,
\ee
where $\lfloor \cdot \rfloor$ denotes the integer part.
Then there exists  a $\varphi$-aperiodic word in $ \Sigma$. 
\end{theorem}

\begin{remark} The condition is satisfied for the following set of parameters:\\
$(1)$ $k\geq 4$, then $\varphi(l)=l$ satisfies \eqref{B}  for $c=2$, \\
$(2)$ $k\geq 5$, then $\varphi(l)=2^l$ satisfies \eqref{B}  for $c=3$, \\
$(3)$ $k\geq 2$, $0<\delta<1$ and $k^{\delta} < c < k$, then there exists $l_0= l_0(k, \delta,c) \in \N_0$ such that 
\be \label{LargeScalePhi}
	\varphi(l)=\begin{cases}
 	 0,  & \text{for } l \leq l_0 \\
 	k^{\delta l},  &\text{for } l>l_0
	\end{cases}
\ee
satisfies  \eqref{B}.
\end{remark}

Note that if a word $w$ is $\varphi$-aperiodic then $R_w(l) > \varphi(l)$ for every $l\in \N_0$ where $R_w$ is the recurrence time introduced in Paragraph $1$.
Theorem \ref{MT1} is hence a corollary of Theorem \ref{TheoremSequence}.

\begin{proof}[Proof of Theorem \ref{MT1}.]
By condition \eqref{AsymptoticCondition}, for every $\varepsilon_0>0$ there exists $l_1=l_1(\varepsilon_0)\in \N$ such that
for all $l\geq l_1$,
\be \nonumber
	\frac{1}{ l }\ln(\varphi(l)) \leq \delta\ln(k)(1+ \varepsilon_0).
\ee
Since $\delta<1$ we let $\e_0>0$ such that $\tilde\delta = (1+\e_0)\delta <1$. 
Then, $\varphi(l) \leq k^{\tilde\delta l}$ for $l\geq l_1$.
If we take $c:= (k - k^{\tilde\delta})/2$ then by  \eqref{LargeScalePhi} there exists $l_2=l_2(k,\tilde\delta)$ such that condition $\eqref{B}$ is satisfied for the function $\bar \varphi(l) := k^{\tilde\delta l}$ for $l> l_2$ and  $\bar \varphi(l)=0$ for  $l\leq l_2$, $l\in \N_0$.  
Theorem \ref{TheoremSequence}  implies the existence of a $\bar \varphi$-aperiodic word $w\in \Sigma$.
Thus, setting $l_0:= \max\{l_1,l_2\}+1$, we have that $\bar \varphi(l)\geq  \varphi(l)$ for all $l\geq l_0$
and the claim follows.
\end{proof}

\begin{remark}
The critical function $\varphi$  for which $\varphi$-aperiodic words cannot exist  
is the function $\varphi(l) = k^{l+1}$. 
The critical exponent  $\ln (k)$ equals the topological entropy of the system $(\Sigma, \bar d, T)$ (see \cite{Walters}) and is optimal. 
To see that  there exists no $w\in \Sigma$ which is $\varphi$-aperiodic for a function $\varphi$ such that $\varphi(l)\geq k^{l+1}-1$ for some $l\in \N_0$, fix a subword $[w(1)\ldots w(1+l)]$ of any $w\in\Sigma$.
Inductively one shows that at each step $1\leq s\leq \varphi(l)$ one has at most $k^{l+1}-s$ possibilities to choose a sub word $[w(1+s)\ldots w(1+s+l)]$ such that $w$ stays $\varphi$-aperiodic. Then, at step $s=k^{l+1}$, there is no choice left such that $w$ is $\varphi$-aperiodic.
\end{remark}

\begin{remark} 
Let  $\Sigma^+(m)= \{w : \{1,\ldots ,m\} \to \mathcal{A}\}$ be the set of words of length $m$ in $\mathcal{A}$ and 
$\mathcal{W}^g(m)\subset \Sigma^+(m)$ be the set of \emph{good} words of length $m$ which satisfy \eqref{Condition} for all $i,s \in \N$ and $l \in \N_0$  such that $i+s+l\leq m$. 
If $\varphi$ satisfies  \eqref{B} with respect to the parameter $c>1$ we will see in the proof of Theorem \ref{TheoremSequence} (see Lemma \ref{AsymptoticGrowth}) that the good words $\mathcal{W}^g(m)$ increase in $m$ by the factor $c$.
Thus,  $\lvert \mathcal{W}^g(m) \rvert \geq  c^m$ which 
is a  lower bound on the asymptotic growth of $\lvert \mathcal{W}^g(m) \rvert$, where $\lvert \cdot \rvert$ denotes its cardinality.
\end{remark}

We may reformulate the critical neighborhood of a periodic point given in \eqref{CriticalNbhd} to the setting of $\varphi$-aperiodicity.
Moreover, since $\cal{P}_T$ is dense in $\Sigma$ we can also give a sufficient condition on $\varphi$-aperiodicity in terms of periodic words.
Therefore, for a non-decreasing unbounded function $ \varphi : \N_0 \to [0, \infty)$,
we define a  discrete form of a right-inverse for $\varphi$ by  $\ell : \N \to \N_0$,
\be \label{DefQuantile}
	 \ell(s) = \min\{ j \in \N_0 : \varphi(j) \geq s\},
\ee
which is also  non-decreasing and unbounded.

\begin{proposition} \label{PeriodicDistance}
Let $\varphi:\N_0 \to [0, \infty)$ be a non-decreasing unbounded function.
If $w\in \Sigma$ is $\varphi$-aperiodic, then for every periodic word $\bar w \in \Sigma$ of period $s$ and  for all $i\in \Z$ we have 
\be\nonumber
	\bar d(T^iw, \bar w)> 2^{- (s+\ell(s))/2}.
\ee
Conversely, if $\bar d(T^iw, \bar w)>2^{-(s+\ell(s)-1)/2}$ for every periodic word $\bar w$ of period $s$ and all $i \in \Z$, then $w$ is $\varphi$-aperiodic.
\end{proposition}

\begin{proof} If $w$ is $\varphi$-aperiodic, $w$ is aperiodic and there exists $m\in \N_0$ such that \\$\bar d(T^iw, \bar w) =2^{-m}$ where we assume $2m\geq s$ (otherwise the first statement follows).
Hence, $[w(i-m)\ldots w(i+m)]=[\bar w(-m)\ldots \bar w(m)]$ and we see that $[w(i-m)\ldots w(i-m+s+(2m-s)]=[w(i-m+s)\ldots w(i+m)]$.
Thus, $s>\varphi(2m-s)$ and $m< (s+\ell(s))/2$ from \eqref{Properties}.

Conversely, assume that $[w(i)\ldots w(i+l)]=[w(i+s)\ldots w(i+s+l)]$ for $s\in \N$, $l\in \N_0$ and let $\bar l= (s+l)/2$ if $s+l$ even and $\bar l=(s+l-1)/2$ if $s+l$ is odd.
Moreover, let $\bar w$ be the periodic word of period $s$ such that $[\bar w(i)\ldots \bar w(i+s-1)]=[w(i)\ldots w(i+s-1)]$.
Thus, $2^{-\bar l}\geq d(T^{i+\bar l}w, T^{i+\bar l}\bar w)> 2^{-(s+\ell(s)-1)/2}$ and we see that $s+\ell(s)-1> 2\bar l \geq s+l-1$.
Hence, $l<\ell(s)$ and from \eqref{Properties} we have $s> \varphi(l)$.
\end{proof}

\begin{remark} Consider the
overlap-free recurrence time $\tilde R^0_w: \N_0 \to \N$ of the initial sub word,
\be  \nonumber
	\tilde R^0_w(l)= \min\{s > l : [w(s)\ldots w(s+l)]= [w(0)\ldots w(l)]\}.
\ee
Clearly, $R_w(l) \leq R^0_w(l) \leq \tilde R^0_w(l)$ for $l\in \N_0$.
Then it follows from  \cite{Weiss} that, since the Bernoulli-shift is ergodic, for $\mu$-almost all $w\in \Sigma$ the limit
\be  \nonumber
	\lim_{l\to \infty} \frac{\ln \tilde R^0_w(l)}{l}
\ee
exists and equals the measure-entropy $h_{\mu}(T)$. 
\end{remark}


\section{Geodesic flow on hyperbolic manifolds}

Let $M$ be a closed $n$-dimensional hyperbolic manifold, that is a compact connected Riemannian manifold without boundary of constant negative curvature $-1$, where $n\geq 2$. 
We denote by $d$ the distance function on $M$ and by $i_M>0$ the injectivity radius.

Let $SM$ be the unit tangent bundle  of $M$ and $d^S$ the Sasaki-distance function on $SM$.
For $v\in SM$ let $\gamma_v : \R \to M$ be the unit speed geodesic such that $\gamma'_v(0)=v$. 
The geodesic flow $\phi^t : SM \to SM$, $t\in \R$, acts on the compact metric space $(SM,d^S)$ by diffeomorphisms, where $\phi^tv = \gamma'_v(t)$. For details and background we refer to \cite{Eberlein}.

A vector $v\in SM$ is \emph{periodic}, if there exists a $t>0$ such that $\phi^tv=v$ and $v$ is \emph{recurrent} if for every $\e>0$ there exists $s>0$ such that $d^S(\phi^sv,v)<\e$.
Denote by $\cal{P}_{\phi}$ and $\cal{R}_{\phi}$ the flow-invariant sets of periodic respectively of recurrent vectors.
Thus if $v\in \cal{R}_{\phi}$ then for a given $t \in \R$, $\e>0$, there exists $s=s(t,\e)$ such that $d^S(\phi^{t+s}v,\phi^tv)<\e$.

We now adjust the definitions of $F$-aperiodic and $\varphi$-aperiodic points to the setting of the geodesic flow.

\begin{definition} Let $F:(0, \infty) \to [0, \infty)$ be a non-increasing function and $s_0>0$ be a constant, called the \emph{minimal shift}.
A vector $v\in SM$ is called \emph{$F$-aperiodic} (with minimal shift  $s_0$) at $t_0\in \R$ if for every $\e>0$, whenever
\be \nonumber
	d^S(\phi^{t_0}v, \phi^{t_0+s}v)<\e
\ee
for some shift $s>s_0$, then $s>F(\e)$. If $v$ is $F$-aperiodic at every time $t_0$ then $v$ is called \emph{$F$-aperiodic} (with minimal shift $s_0$).
\end{definition}

Note that in contrast to the discrete setting in Section $2$ (where $s\in \N$, i.e. $s\geq 1$) we now have to specify the additional parameter $s_0$, since $d^S(\phi^{t_0}v, \phi^{t_0+s}v)=s$ for $s$ small enough.

We also have to generalize the notion of $\varphi$-aperiodicity. All geodesics will be assumed to be unit speed.
Note that as in the case of the Bernoulli-shift, two vectors in the Sasaki-distance are very close if and only if the trajectories of the corresponding geodesics are close (in the Riemannian distance) to each other for a long time. Thus we may reformulate $\varphi$-aperiodicity in terms of the \emph{fellow traveller length}.

Herefore we introduce a second parameter, the \emph{distance constant} $\e_0>0$.

\begin{definition} \label{ContinuousDef} 
Let $\varphi : [0, \infty) \to [0, \infty)$ be a non-decreasing function, let $0<\varepsilon_0<i_M$  
and $s_0\geq\e_0$.
A  geodesic  $\gamma : \R \to M$  is called \emph{$\varphi$-aperiodic}  at \emph{time} $t_0\in \R$ if for every \emph{length} $l >\e_0 $, whenever
\be \nonumber
	d(\gamma(t_0 + t) , \gamma(t_0 +s +t)) < \varepsilon_0 \ \ \ \ \text{ for all } 0\leq t \leq l
\ee
for some \emph{shift} $s>s_0$, then $s > \varphi(l)$.
If $\gamma$ is $\varphi$-aperiodic at every time $t_0$, it is called \emph{$\varphi$-aperiodic} (with parameters $(s_0,\e_0)$).
\end{definition}

The geodesic flow on compact hyperbolic manifolds is ergodic with respect to the Liouville measure $\mu$ (on the Borel-$\sigma$-algebra of $SM$). 
A systole of $M$ has length $2i_M$ which equals the systolic period. For a non-decreasing function $\varphi$ let  $\cal{F}_{\phi}$ be the set of $\varphi$-aperiodic geodesics (with respect to  $(s_0,\e_0)$), which is invariant under the geodesic flow $\phi^t$.
Since $\mu$ is positive on open sets, one can show as in Lemma \ref{Measure}, 
that the set $\cal{F}_{\phi}$ is of  zero $\mu$-measure if and only if $\varphi$ is not bounded by either $s_0$ or $2i_M-\varepsilon_0$.

The main result of this section is the following, which will be proved in the Section \ref{Paragraph2}.

\begin{theorem} \label{MainThm}
Assume that $i_M > \ln(2) $ and let $\e_0>0$ such that $\ln(2) + \e_0 <i_M$. Let 
\be \nonumber
	\varphi_{\delta}(l) = e^{\delta(n-1)l },
\ee
where $0<\delta<1$.
Then there exists  a minimal length $l_0=l_0(\delta, i_M,n,\e_0)$ and a geodesic $\gamma: \R \to M$ which satisfies for every $t_0\in\R$ and all $l \geq l_0 $, whenever
\be \label{MTC}
	d(\gamma(t_0 + t) , \gamma(t_0 +s +t) < \varepsilon_0 \ \ \ \ \text{ for all } 0\leq t \leq l
\ee
for some shift $s> \e_0$, then $s > \varphi_{\delta}(l)$.
\end{theorem}

Note  that for $\e_0=i_M/2$, if a geodesic $\gamma: \R \to M$ satisfies \eqref{MTC} then $R_{\gamma}(l) \geq \varphi_{\delta}(l)$ for all $l\geq l_0$, where $R_{\gamma}$ is the recurrence time introduced in Paragraph $1$.
Theorem \ref{MT} is hence a corollary of Theorem \ref{MainThm}.

\begin{proof}[Proof of Theorem \ref{MT}.]
By \eqref{AsymptoticCondition2}  there exists for every $\tau>0$  some $l_1=l_1(\tau)\geq 0$ such that
for all $l\geq l_1$ we have
\be \nonumber
	\varphi(l) \leq e^{(1+  \tau)(n-1) \delta l}.
\ee

Since $\delta<1$ we let $\tau_0>0$ such that $\bar \delta:=(1+  \tau_0)\delta<1$.
From Theorem \ref{MainThm} for $\e_0=i_M/2$, there exists an $l_2 = l_2(\bar \delta,  i_M,n)$ 
and a geodesic geodesic $\gamma : \R \to M$ such that for every $t_0 \in \R$ and $l \geq l_2$,  whenever
\be \nonumber
	d(\gamma(t_0 + t), \gamma(t_0+s +t) ) < \frac{i_M}{2} \ \ \ \text{ for all } \  0\leq t \leq l,
\ee
for some shift $s > i_M/2$, then  $s > e^{\bar \delta(n-1)l}$. 
If we set $l_0 := \max\{l_1,l_2\}$ then $s>e^{\bar \delta(n-1)l} \geq \varphi(l)$ whenever $l\geq l_0$ and the proof is finished.
\end{proof}

In order to prove Theorem \ref{MainThm} we discretize our geodesics.
Therefore we need a third parameter, the \emph{discretization constant} $r_0>0$.
To a geodesic $\gamma: \R \to M$ we consider the \emph{discrete} geodesic
\be \nonumber
	\bar \gamma: \Z \to M, \ \ \ \ \bar \gamma(i) := \gamma(i \cdot r_0).
\ee

\begin{definition}\emph{(Discrete Definition)} \label{DiscreteDef} 
Let $\bar \varphi : \N_0 \to [0, \infty)$ be a non-decreasing function and let the parameters $(\bar s_0, \bar \e_0, r_0)$ be given where $\bar s_0\in \N_0$, $0<\bar \e_0 < i_M$ and $0<r_0< \bar \e_0$.
A discrete  geodesic  $\bar \gamma : \Z \to M$  is called \emph{$\bar \varphi$-aperiodic} at time  $i\in\Z$ if for $l \in\N$, whenever
\be 
\label{ConditionVarphi}
	d(\bar \gamma( i+j) ,\bar \gamma( i+s+j)) < \bar \varepsilon_0 \ \ \ \ \text{ for all } j \in \{0,\ldots ,l\}
\ee
for some shift $s>\bar s_0$, then $s > \bar \varphi(l)$.
$\bar \gamma$ is called \emph{$\bar \varphi$-aperiodic}  (with parameters $(\bar s_0, \bar \e_0, r_0)$)   if it is $\bar \varphi$-aperiodic at every time $i\in\Z$.
\end{definition}

Note that, given a $\bar \varphi$-aperiodic geodesic $\bar \gamma : \Z \to M$ (with the parameters  $(\bar s_0, \bar \e_0, r_0)$), the corresponding geodesic $\gamma : \R \to M$ is continuously $\varphi$-aperiodic in the following way.

\begin{lemma} \label{ContCond}
For a non-decreasing function $\bar \varphi : [0, \infty) \to [0, \infty)$ and the parameters  $(\bar s_0, \bar \e_0, r_0)$  let 
$ \bar \gamma: \Z \to M$ be a $\bar \varphi \lvert_{\N_0}$-aperiodic geodesic. 
For $r_0 \leq l\in \R$, define
\be \nonumber
	 \varphi (l) := r_0 \cdot \bar \varphi(\frac{l-r_0}{r_0})  -r_0.
\ee
Then $\gamma$ is  $\varphi$-aperiodic with respect to  the minimal shift $s_0= (\bar s_0+1)r_0$  and the distance constant $\varepsilon_0 = \bar \e_0-r_0>0$.

Conversely, if $\gamma: \R \to M$ is $\varphi$-aperiodic with parameters $(s_0, \e_0)$ then for $r_0<\e_0$, let
\be \nonumber
	\bar  \varphi(l):=\varphi(l \cdot r_0) /r_0.
\ee
Then $\bar \gamma: \Z \to M$ is $\bar \varphi$-aperiodic with parameters $(\lceil s_0/r_0 \rceil, \e_0, r_0)$.
\end{lemma}

\begin{proof}
For $t_0 \in \R$, $L  \geq r_0$ and $s > (\bar s_0+1)r_0$ assume that $d(\gamma(t_0 + t), \gamma(t_0+s+t) ) < \varepsilon_0$ for all $0\leq t \leq L$.
If  we set $i:= \lceil \frac{t_0}{r_0} \rceil$ and $i+\bar s:= \lceil \frac{t_0+s}{r_0} \rceil$ whereas $ l := \lfloor \frac{L}{r_0} \rfloor$, we have $i$, $ l \geq 1$ and $\bar s>\bar s_0$.
Then, since $\e_0 = \bar \e_0-r_0<  i_M$ and the distance function is locally convex, one checks by the triangle inequality that $d(\bar \gamma(i), \bar \gamma(i+\bar s) ) <\bar  \varepsilon_0$ and $d(\bar \gamma(i+l), \bar \gamma(i+\bar s+ l) ) < \bar \varepsilon_0$. In particular,
$d(\bar \gamma(i+j), \bar \gamma(i+\bar s +j) ) < \bar \varepsilon_0$ for all $0 \leq j \leq l$. Thus, $\bar s> \bar \varphi( l)$ so that
\be	 \nonumber
	s \geq (\bar s-1)r_0 > (\bar \varphi(l) -1 )r_0 \geq \big(\bar \varphi(\frac{L}{ r_0} -1) -1\big) r_0 =  \varphi(L)
\ee
since $ (l +1)r_0\geq L$.
This finishes the first part of the Lemma. The second part  follows analogously.
\end{proof}

In terms of Lemma \ref{ContCond} we are left with stating the existence theorem for discrete $\bar \varphi$-aperiodic geodesics.
Recall that  for an unbounded function $\bar \varphi$  we  defined its discrete right-inverse $\bar \ell : \N \to \N_0$ in 
\eqref{DefQuantile} which is also non-decreasing and unbounded.

\begin{theorem}  \label{TheoremGeodesic} 
Let $\bar \varphi:\N_0 \to [0,\infty)$ be a non-decreasing, unbounded function.
Assume that $\ln(2) <r_0< \bar \e_0 < i_M$ and $\bar s_0 \in \N_0$ such that for all $l\geq \bar s_0$,
\be \label{D}
\begin{array}{c}
\lfloor \bar \varphi(l) \rfloor > l, \ \ \ \text{ and } \ \ \ \ 
\bar \ell(\bar s_0) \geq 1,
\end{array}
\ee
and moreover, that there exists a constant $c\in(1,2^{n-1})$ such that 
\be \label{C}
  2^{n-1} - \bar c \cdot \sum_{l=\bar \ell(\bar s_0)}^{\infty}   \frac{ \lfloor\bar  \varphi(l) \rfloor - \lfloor\bar \varphi(l-1) \rfloor}{c^l} \geq c,
\ee
where $\bar c$ is an explicit constant depending only on $n$ and $i_M$.
Then  there exist a $\bar \varphi$-aperiodic geodesic $\gamma: \Z \to M$ with the parameters $(\bar s_0, \bar \e_0,r_0)$.
\end{theorem}

\begin{remark} Since $\bar \ell$ is unbounded, condition \eqref{C} depends again essentially on the convergence of the sum in  \eqref{C}.
For instance, let  $\delta \in (0,1)$ and define  $\bar \varphi(l) =2^{\delta(n-1)l}$ and let $c\in (2^{\delta(n-1)}, 2^{n-1})$.
Then, since  $\bar \ell(s)=\lceil \frac{1}{\delta(n-1)\ln(2)} \ln(s) \rceil$  for $s\geq0$, 
there exists a minimal shift $\bar s_0= \bar s_0(n, \delta,\bar c,c)$ such that  \eqref{D} and \eqref{C} are satisfied.

The constant $\bar c$ of condition \eqref{C} can in fact be sharped to be also dependent on $\bar s_0$, in which case it is strictly decreasing in $\bar s_0$.
It will be explicitly defined in the proof of claim \ref{Claim}.
We may give a rough upper bound of $\bar c$ which is independent of $\bar s_0$ by
\be \label{ComputeConstant}
	\bar c \leq \lceil \big(3 \cosh(i_M) \sqrt{n+1} \big)^{n-1} \rceil \lceil \frac{\int_0^{5i_M+4\ln(\sqrt{n+1}/2)} \sinh(t)^{n-1}dt}{\int_0^{i_M/2} \sinh(t)^{n-1}dt} \rceil.
\ee

The lower bound $\ln(2)$ on the injectivity radius is necessary for the proof. However we believe that the result should be valid without this bound.
Moreover, a version of Theorem \ref{TheoremGeodesic} remains true for $M$ a closed $n$-dimensional Riemannian manifold of negative sectional curvature.
\end{remark}

\begin{remark} Again, the critical function $\varphi$ such that $\varphi$-aperiodic geodesics might or might not exist seems to be the function $\varphi(s) =e^{(n-1)s}$  and the critical exponent $n-1$ equals the  topological entropy of $(SM, \phi^t)$.

Lemma \ref{UpperBound} gives an upper bound on the growth rate of non-increasing functions $F:(0, \infty) \to (0, \infty)$ for which $F$-aperiodic geodesics can exist. 
In fact, since $SM$ is a $(2n-1)$-dimensional manifold, its box dimension is $2n-1$.
Discretizing $\phi^t$ by the time $t_0$-map $\phi^{t_0}$ where $t_0= t_0(i_M)>0$ is sufficiently small, 
 gives the upper bound 
\be \nonumber
	\limsup_{\e\to 0}\frac{ \ln(F(\varepsilon))}{\ln (2/\e)} \leq 2n-1.
\ee
\end{remark}

\begin{remark}
For a closed geodesic $\alpha : \R\to M$, let $\cal{N}_{\varepsilon_0}(\alpha)$ be the (closed) $\e_0/2$-neighborhood of $\alpha$ in $M$, where $\e_0>0$ sufficiently small. 
When a geodesic $\gamma:\R \to M$ enters $\cal{N}_{\varepsilon_0}(\alpha)$  at time $t_0$ let $\mathfrak{p}_{\alpha}(\gamma, t_0)$ be the \emph{penetration length} of $\gamma$ in $\alpha$ at time $t_0$, that is,  the maximal length $L\in [0, \infty]$ of an interval $I$, $t_0\in I$, such that  $\gamma(t) \in \cal{N}_{\varepsilon_0}(\alpha)$ for all $t\in I$. Set $\mathfrak{p}_{\alpha}(\gamma, t_0)=0$ if 
$\gamma(t_0)\not \in \cal{N}_{\varepsilon_0}(\alpha)$. 
Then by \cite{Paulin}, for $\mu$-almost every $v\in SM$ the limit 
\be \label{Asymptotic2}
	\limsup_{t\to \infty} \frac{\mathfrak{p}(\gamma_v(t))}{\ln(t)}
\ee 
exists and equals $1/(n-1)$.

Moreover, the penetration length reflects the \emph{depth} in which $\gamma$ enters the neighborhood $\cal{N}_{\varepsilon_0}(\alpha)$.
The study of depths or penetration lengths in an adequate convex set of negatively curved manifolds, such as the $\e$-neighborhood of totally geodesic embedded submanifold  or the cusp-neighborhood of a finite-volume hyperbolic manifold, leads to the theory of diophantine approximation in negatively curved manifolds; see for instance \cite{Haas,Hersonsky,Paulin, Parkkonen,Parkkonen2,Patterson,Sullivan,Velani} to give only a short and incomplete list.
In general, a sequence of depths or penetration lengths and times of $\gamma$ in these convex sets reflects "how well $\gamma$ is approximated", where $\gamma$ is called \emph{badly approximable} if any such sequence is bounded; see \cite{Hersonsky,Paulin}.

Now, let $\gamma$ be a $\varphi$-aperiodic geodesic ($\varphi$ unbounded) with respect to the parameters $s_0$ and $\e_0$ and let $\alpha$ be \textbf{any} closed geodesic in $M$. 
Then, it can be seen that the penetration lengths of  $\gamma$ in $\cal{N}_{\e_0}(\alpha)$ are bounded by a constant depending only on $\varphi$, $\e_0$ and the length of $\alpha$ (and  $s_0$ respectively). 
Therefore, the notion of $\varphi$-aperiodicty is linked to bad approximation; recall also Example \ref{Ex1}.
In particular, the limit of \eqref{Asymptotic2} equals $0$ for $\gamma$.
\end{remark}


\section{Proofs} \label{Paragraph2}
Let $\varphi : \N_0 \to [0, \infty)$ be a non-decreasing unbounded function. 
Recall the definition of the function  $\ell : \N \to \N_0$ given by
\be 
\nonumber
	\ell(s) = \min\{ j \in \N_0 : \varphi(j) \geq s\},
\ee
see \eqref{DefQuantile}.
The following properties hold: $\ell$ is non-decrasing and for $s$ and $l \in \N_0$, we have 
\be \label{Properties}
\begin{array}{c}
		\varphi( \ell(s)) \geq s, \\[1mm]
	l < \ell(s) \iff \varphi(l) < s, \\[1mm]
	l \geq \ell(s) \iff \varphi(l) \geq s.
\end{array}
\ee
\begin{proof} For the first property, clearly $\varphi( \min\{ j: \varphi(j) \geq s \} )  \geq s$.
Let $l< \ell(s)$ and assume $s \leq \varphi(l)$. Then $\ell(s) = \min\{ j: \varphi(j) \geq s \} \leq l$; a contradiction.
If $s> \varphi(l)$ then $\ell(s)= \min\{ j: \varphi(j) \geq s \} > l$ and if $\varphi(l) \geq s$ then $\ell(j)= \min\{ j: \varphi(j) \geq s \} \leq l$.
Also, if $l\geq \ell(s)$ then $\varphi(l) \geq  \varphi(\ell(s)) \geq s$.
\end{proof}

\subsection{Proof of Theorem \ref{TheoremSequence}.}

Recall that $\Sigma^+(m)= \{w : \{1,\ldots ,m\} \to \mathcal{A}\}$ is the set of words of length $m-1$.
We consider $\Sigma^+(m)$ to be a subset of $\Sigma^+=\cal{A}^{\N}$ (for example, by extending an element $w\in \Sigma^+(m)$ to an element $\bar w\in \Sigma^+$ by setting $\bar w(i) = a$ for all $i>m$, where $a\in \cal{A}$ is fixed).

\begin{definition} Let $m\in\N$. $w\in \Sigma^+(m)$ is called \emph{$\varphi$-aperiodic} if for all $i, s \in\N$ and $l \in\N_0$ such that $i+s+l\leq m$ whenever 
\be	\nonumber
	[w(i)\ldots w(i+l)]= [w(i+s)\ldots (w(i+s+l)]
\ee
we have $s>\varphi(l)$.
\end{definition}

Let $l_0 := \min\{j\in \N_0 \cup \{-1\} : \varphi(j+1)\neq 0\}$ and note that $\ell(s)>l_0$ for all $s\in \N$.
For $m\in \N$, define the \emph{admissible set}  by
\be \nonumber 
	A(m) := \{ (i,s) \in \N\times\N: i+s+\ell(s)=m\},
\ee
if $m\geq m_0 := 2+\ell(1)>2+l_0$ and let $A(m)$ be empty for $m< m_0$.
Then, for $(i,s)\in A(m)$ where $m\geq m_0$, we define the sets
\be\nonumber
	C_{is} := \{w \in \Sigma^+(m): [w(i)\ldots w(i+\ell(s))] \neq [w(i+s)\ldots w(i+s+\ell(s))]\},
\ee
called \emph{conditions}.

\begin{remark} Note that $s> \varphi(\ell(s)-1)$ for $\ell(s)>0$ but $s\leq \varphi(\ell(s))$. Therefore $\ell(s)$ determines the critical length of a given shift $s$ with respect to $\varphi$. 
\end{remark}

For $w \in \Sigma^+(m)$ and $1\leq n \leq m$ let $w\lvert_n := [w(1)\ldots w(n)] \in \Sigma^+(n)$. 
This leads to the reformulation of $\varphi$-aperiodic words: 

\begin{lemma}\label{Reformulation}
For $m<m_0$ every word $w\in \Sigma^+(m)$ is $\varphi$-aperiodic.
For $m\geq m_0$, a word $w\in \Sigma^+(m)$ is $\varphi$-aperiodic if and only if for all  $n\leq m$ and all $(i,s)\in A(n)$ we have $w\lvert_n \in C_{is}$.
\end{lemma}

\begin{proof} 
First, let $m<m_0$. Then for every $i$,$s\in \N$, $l\in \N_0$ such that $i+s+l\leq m<2+\ell(1)$ we have in particular $l<\ell(1)$.
Equivalently, $\varphi(l)<1$ so that $s>\varphi(l)$ and every word $[w(1)\ldots w(m)]$ follows to be $\varphi$-aperiodic. 

Now let $m\geq m_0$.
Let $w$ be $\varphi$-aperiodic and assume $w\lvert_n \not \in C_{is}$ for some $i$ and $s$ in $\N$ such that $i+s+\ell(j)=n\leq m$.
Then 
\be
	 [w(i)\ldots w(i+\ell(s))] = [w(i+s)\ldots w(i+s+\ell(s))] \nonumber
\ee
and by \eqref{Condition}, we have $s>\varphi(\ell(s))$; a contradiction to $\varphi(\ell(s)) \geq s$. 

Conversely, assume that $w$ is not $\varphi$-aperiodic. Then there are $i$, $s \in \N$ and $l \in \N_0$ such that $i+s+l \leq m$ and
\be
	 [w(i)\ldots w(i+l)] = [w(i+s)\ldots w(i+s+l)] \nonumber
\ee
 with $s \leq \varphi(l)$.
This implies that $\ell(s) \leq l$ and in particular 
\be
	[w(i)\ldots w(i+\ell(s))] = [w(i+s)\ldots w(i+s+\ell(s))]. \nonumber
\ee
Hence, it follows that $w \lvert_n \not \in C_{is}$ since $i+s+\ell(s) = n\leq m$ so that $(i,s)\in A(n)$.
\end{proof}

\noindent Note that by the same arguments as in the previous proof, a word $w\in \Sigma^+$ is $\varphi$-aperiodic if and only if for all  $ n\geq m_0$ and all $(i,s)\in A(n)$ we have $w\lvert_n \in C_{is}$.

 For $m \in \N$ such that $m\geq m_0$ the set of good words of length $m$ is therefore given by 
\be\nonumber
	\mathcal{W}^g(m) = \{w \in \Sigma^+(m) : w\lvert_n \in C_{is} \text{ for all } (i,s)\in A(n) \text{ where } n \leq m \},
\ee
and by $\mathcal{W}^g(m) = \Sigma^+(m)$ otherwise.
Let 
\be\nonumber
	\mathcal{C}_m = \{C_{is}: (i,s)\in A(m) \} 
\ee
 be the set of conditions at place $m$ which is empty if and only if $m< m_0$. 
Clearly, if $w\in \mathcal{W}^g(m)$ then $w\lvert_n \in \mathcal{W}^g(n)$ for $n\leq m$.

\begin{lemma}\label{Iteration}
For $m \in \N$, 
\be\nonumber
	\lvert \mathcal{W}^g(m+1) \rvert \geq  k \cdot \lvert \mathcal{W}^g(m) \rvert - \sum_{C_{is} \in \mathcal{C}_{m+1}} \lvert \mathcal{W}^g(i+s-1) \rvert 
\ee
\end{lemma}

\begin{proof}
If $m+1<m_0$ then $\mathcal{C}_{m+1}$ is empty and the claim follows.
Hence let $m+1\geq m_0$. 
Set $L = \{ w \in \Sigma^+(m+1) : w\lvert_m \in \mathcal{W}^g(m)\}$.
Then
\be \nonumber
	\mathcal{W}^g(m+1) = 
	L \cap \big(  \bigcap_{ C_{is} \in \mathcal{C}_{m+1}} C_{is}  \big) = 
	L \setminus \big( \bigcup_{ C_{is} \in \mathcal{C}_{m+1}} (L\cap C_{is}^C)  \big),
\ee
where $C_{is}^C$ denotes the complement of $C_{is}$.
Fix some condition $C_{is} \in \mathcal{C}_{m+1}$. 
Since $\lvert L \rvert =  k  \cdot \lvert \mathcal{W}^g(m) \rvert$ the Lemma follows from the following claim.
\end{proof}

\begin{claim} $\lvert L \cap C_{is}^C \rvert \leq \lvert \mathcal{W}^g(i+s-1) \rvert$.
\end{claim}

\begin{proof} If $Q:= \{ w\lvert_{i+s-1} \in \Sigma^+(i+s-1) : w\in L \}$ then clearly $\lvert Q \rvert \leq  \lvert \mathcal{W}^g(i+s-1) \rvert$.
Decompose $L$ into $L = \cup_{q \in Q}L_q$ where $L_q = \{ w \in L : w\lvert_{i+s-1} = q\}$.
By definition, different elements in $L_q$ have different subwords $[w(i+s)\ldots w(m+1)]$ and moreover 
\be \nonumber
	L\cap C_{is}^C= \{w\in L : [w(i)\ldots w(i+\ell(s)] = [w(i+s)\ldots w(m+1)]\}.
\ee
Hence, if $s> \ell(s)$ then an element $w$ of  $L_q$, which is also in $C_{is}^C$, is uniquely determined by $q$, that means, $w$ is of the form $w \lvert_{i+s-1} = q$ and  
\be \nonumber
	[w(i+s)\ldots w(m+1)]= [q(i)\ldots q(i+\ell(s))].
\ee 
If $s\leq \ell(s)$ then one inductively checks that a word $w$ in $L_q \cap C_{is}^C$ is of the form $w \lvert_{i+s-1} = q$, 
\be \nonumber
\begin{array}{lcl} 
	[w(i+js)\ldots w(i+(j+1)s-1)] &=& [w(i+(j-1)j)\ldots w(i+js-1)]  = \ldots  =\\
	 &=& [w(i)\ldots w(i+s-1)] =  [q(i)\ldots q(i+s-1)]
\end{array}
\ee
for $1\leq j \leq j_0$ where $j_0$ is the maximal $j$ such that $i + (j+1)s -1 \leq m+1$, and 
\be \nonumber
	[w(i + (j_0+1) j)\ldots w(m+1) ] = [q(i)\ldots q(m+1 - (i+(j_0+1)s))],
\ee
if $i+(j_0+1)s <m+1$.
Again, $w$ is uniquely determined by $q$. Hence in both cases, $\lvert L_q \cap C_{is}^C \lvert \leq 1$ and therefore 
\be	 \nonumber
	\lvert L \cap C_{is}^C \rvert \leq \lvert Q \rvert \leq \lvert \mathcal{W}^g(i+s-1) \rvert
\ee
which proves the claim.
\end{proof}

The above Lemma yields the following crucial estimate:

\begin{lemma} \label{CountConditions} 
For $m\in \N$,
\be
	\lvert \mathcal{W}^g(m+1) \rvert \geq 
	\big( k- \lfloor \varphi(0)\rfloor \big) \lvert \mathcal{W}^g(m) \rvert - \sum_{j=1}^m \big( \lfloor \varphi(j) \rfloor - \lfloor \varphi(j-1) \rfloor\big) \lvert \mathcal{W}^g(m-j) \rvert.
\ee
\end{lemma}

\begin{proof}
For $0\leq j \leq m$ let
\be \label{Hs}
	H_j= \{ C_{is} \in \mathcal{C}_{m+1}: i+s-1 = m-j \},
\ee
possibly empty.
If $C_{is}\in H_j$ then $i+s+\ell(s)=m+1$ and $i+s-1 = m-j$; hence $\ell(s) = j$.
Therefore, $\lvert H_j \rvert \leq \lvert \{ s : \ell(s) =j\} \rvert$.
We have $\ell(s) \leq j$ if and only if $s \leq \varphi(j)$  and thus
\be \nonumber
	\lvert \{ s : \ell(s)\leq j \} \rvert = \lvert \{ s : s \leq \varphi(j) \} \rvert = \lfloor \varphi(j) \rfloor.
\ee
For $j \geq 1$ this implies that
\be  \nonumber
\begin{array}{lcl}
	\lvert H_j \rvert &\leq& \lvert \{ s : \ell(s) =j\} \rvert = \lvert  \{ s: \ell(s)\leq j \} \setminus \{ s: \ell(s) \leq j-1 \} \rvert  \\
	&=& \lfloor \varphi(j) \rfloor - \lfloor \varphi(j-1) \rfloor.
\end{array}
\ee  
Moreover, 
\be \nonumber
	\lvert \{ s : \ell(s) =0\} \rvert = \lvert \{ s \in \N_0: \varphi(0) \geq s \} \rvert  = \lfloor \varphi(0) \rfloor.
\ee 
Lemma \ref{Iteration} concludes the proof.
\end{proof}

Finally we show the existence of a $\varphi$-aperiodic word in $\Sigma^+$.

\begin{lemma} \label{AsymptoticGrowth}
If condition \eqref{B} is satisfied, then $\lvert \mathcal{W}^g(m) \rvert  \geq c^{m}$. In particular, there exists a $\varphi$-aperiodic word in $\Sigma^+$.
\end{lemma}

\begin{proof}
For $m+1 < m_0$ we have that $\lvert \mathcal{W}^g(m+1) \rvert = k^{m+1} \geq c^{m+1}$.
For $m+1\geq m_0$  assume that $\lvert \mathcal{W}^g(n) \rvert \geq c \cdot \lvert \mathcal{W}^g(n-1) \rvert$ for all $n\leq m$. 
Then, by the previous Lemma,
\be \label{Formula}
\begin{array}{lcl} 
	\lvert \mathcal{W}^g(m+1) \rvert &\geq&
	
	 ( k- \lfloor \varphi(0)\rfloor) \lvert \mathcal{W}^g(m) \rvert - \sum_{j=1}^m ( \lfloor \varphi(j) \rfloor - \lfloor \varphi(j-1) \rfloor) \lvert \mathcal{W}^g(m-j) \rvert \\[1.5mm]

	&\geq&  ( k- \lfloor \varphi(0)\rfloor) \lvert \mathcal{W}^g(m) \rvert -  \sum_{j=1}^m \frac{ \lfloor \varphi(j) \rfloor - \lfloor \varphi(j-1) \rfloor}{c^j}\lvert \mathcal{W}^g(m) \rvert \\[1.5mm]
	
	&\geq&  \Big( k-  \lfloor \varphi(0)\rfloor- \sum_{j=1}^{\infty} \frac{ \lfloor \varphi(j) \rfloor - \lfloor \varphi(j-1) \rfloor}{c^j} \Big) \lvert \mathcal{W}^g(m) \rvert 
	
	\geq c \cdot \lvert \mathcal{W}^g(m) \rvert,
\end{array}
\ee
where we used condition \eqref{B} in the last inequality.
Now Lemma \ref{Reformulation} implies the existence of a $\varphi$-aperiodic word in $\Sigma^+$.
\end{proof}

Given a $\varphi$-aperiodic word $w\in \Sigma^+$ and a letter $a \in \cal{A}$, extend $w$ to a word $\ldots aaaw =: \bar w \in \Sigma$ (in the obvious way).
Consider the sequence $\{T^n\bar w\}_{n\in \N}$ in the compact space $\Sigma$ and let $w_0$ be an accumulation point.
Note that from the definition of the metric $\bar d$, a sequence $w^n$ in $\Sigma$ converges to a word $w_0 \in\Sigma$ if and only if for every $l\in \N_0$ there exists $N\in \N$ such that  $[w^n(-l)\ldots w^n(l)] = [w_0(-l)\ldots w_0(l)]$ for every $n\geq N$. It therefore follows that $\varphi$-aperiodicity is a closed condition (as showed similarly in Lemma \ref{ClosedCond}).
Since every $T^n\bar w$ is $\varphi$-aperiodic starting at time $-(n-1)$, $w_0$ is a $\varphi$-aperiodic word in $\Sigma$. 
This proves Theorem \ref{TheoremSequence}.


\subsection{Proof of Theorem \ref{TheoremGeodesic}.}
Recall that $M$ is a closed hyperbolic manifold of dimension $n\geq2$ and we have $\ln(2) <r_0< \bar \e_0 < i_M$. Moreover $\bar \varphi:\N_0\to [0, \infty)$ is a non-decreasing unbounded function for which conditions \eqref{D} and \eqref{C} are satisfied with respect to the given minimal shift $\bar s_0\in \N_0$.

A reference for the following is given by \cite{Eberlein,BGS}.
Let $\H^n$ be the $n$-dimensional hyperbolic upper half-space model where $d$ denotes the hyperbolic distance function on $\H^n$.
Let $\Gamma$ be the discrete, torsion-free subgroup of the isometry group of $\H^n$ identified with the fundamental group $\pi_1(M)$ of $M$ acting cocompactly on $\H^n$ such that the manifold $\Gamma \backslash \H^n$ with the induced smooth and metric structure is isometric to $M$. 
Let $\pi : \H^n \to  \Gamma \backslash \H^n \cong M$ be the projection map. 
Assume all geodesic segments, rays or lines  to be parametrized by arc length and identify their images with their point sets in $\H^n$.
Let $\partial_{\infty}\H^n$ be the set of equivalence classes of asymptotic rays in $\H^n$  which we identify with the set   $\R^{n-1} \cup \{\infty\}$, where $\bar \H^n - \{\infty\} = \H^n \cup \R^{n-1}$ is equipped with the induced Euclidean topology. 
If $\gamma$ is a ray in $\H^n$ we will simply write $\gamma(\infty)$ for the corresponding point in $\partial_{\infty}\H^n$.
For any two points $p$ and $q$ in $\bar \H^n$ denote  by $[p,q]$ the geodesic segment, ray or line in $\H^n$ - depending on if  $p, q \in \H^n$,  $p\in \H^n$ and $q \in \partial_{\infty}\H^n$, or $p,q  \in \partial_{\infty}\H^n$ respectively - connecting $p$ and $q$.

For $t\in \R$ let $H_t:= \R^{n-1}\times \{e^{-t}\} \subset \H^n$. This equals the horosphere based at $\infty$ through the point $\gamma(t)$ of the unit speed geodesic $\gamma(t)=(0,e^{-t})$. Let $h_t$ be the induced length metric on $H_t$ with respect to $d$  .
The geometry of horospheres in the hyperbolic space is well-known; see for instance \cite{Heintze} for the following facts.
$(H_t, h_t)$ is a complete and flat metric space, isometric to the $(n-1)$-dimensional Euclidean space. 
If $\gamma_i : \R \to \H^n$ with $\gamma_i(0) \in H_0 $ , $i=1,2$, are two geodesic lines in $\H^n$ with $\gamma_1 (-\infty) = \gamma_2(-\infty)=\infty$ and $\gamma_1(0)$, $\gamma_2(0)$ in the same horosphere, let $\mu(t) := h_t( \gamma_1(t), \gamma_2(t))$.
Then,  for $t \geq 0$,
\be  \label{Heintze}
	 \mu(t) = e^t \mu(0).
\ee
Moreover,  for two points $p,q$ in the same horosphere $H_t$ we have
\be \label{Heintze2}
	 h_t(p,q) = 2 \sinh (d(p,q)/2).
\ee

Now let  $\tau>0$ such that the discretization constant satisfies $r_0= \ln 2 + \tau$. Let $R>0$ be a fixed length, say $R=1$. Define $Q$ to be an isometric copy of a closed $(n-1)$-dimensional cube $[-R/2, R/2]^{n-1}$ of edge lengths $R$ in the Euclidean space $\E^{n-1}$ and contained in the horosphere $H_0$.
Starting with the cube $Q$ as a reference, we inductively shed shadows in the horospheres $H_{m r_0}$, $m \in \N$, as follows:

\begin{definition} Given two disjoint sets $S$ and $S'$ in $\bar \H^n$, the set  $\mathcal{S}(S;S'):=\{q \in S' : S \cap [\infty,q] \neq \emptyset \}$ is called the \emph{shadow of $S$ in $S'$} (with respect to $\infty$).
\end{definition}

\noindent By \eqref{Heintze}, the shadow $\mathcal{S}(Q; H_{r_0})$ of $Q$ is an isometric copy of a closed $(n-1)$-dimensional cube  of edge lengths $e^{r_0} R = (2 + e^{\tau})R$,  contained in $H_{r_0}$.
Hence, there exist $2^{n -1}$ disjoint isometric copies $Q_j$, $j \in \{1,\ldots ,2^{n-1}\}$, of $Q$ in $\mathcal{S}(Q; H_{r_0})$; see Figure \ref{HoroInc}.

\begin{center} \label{HoroInc}
 \includegraphics[scale=0.50]{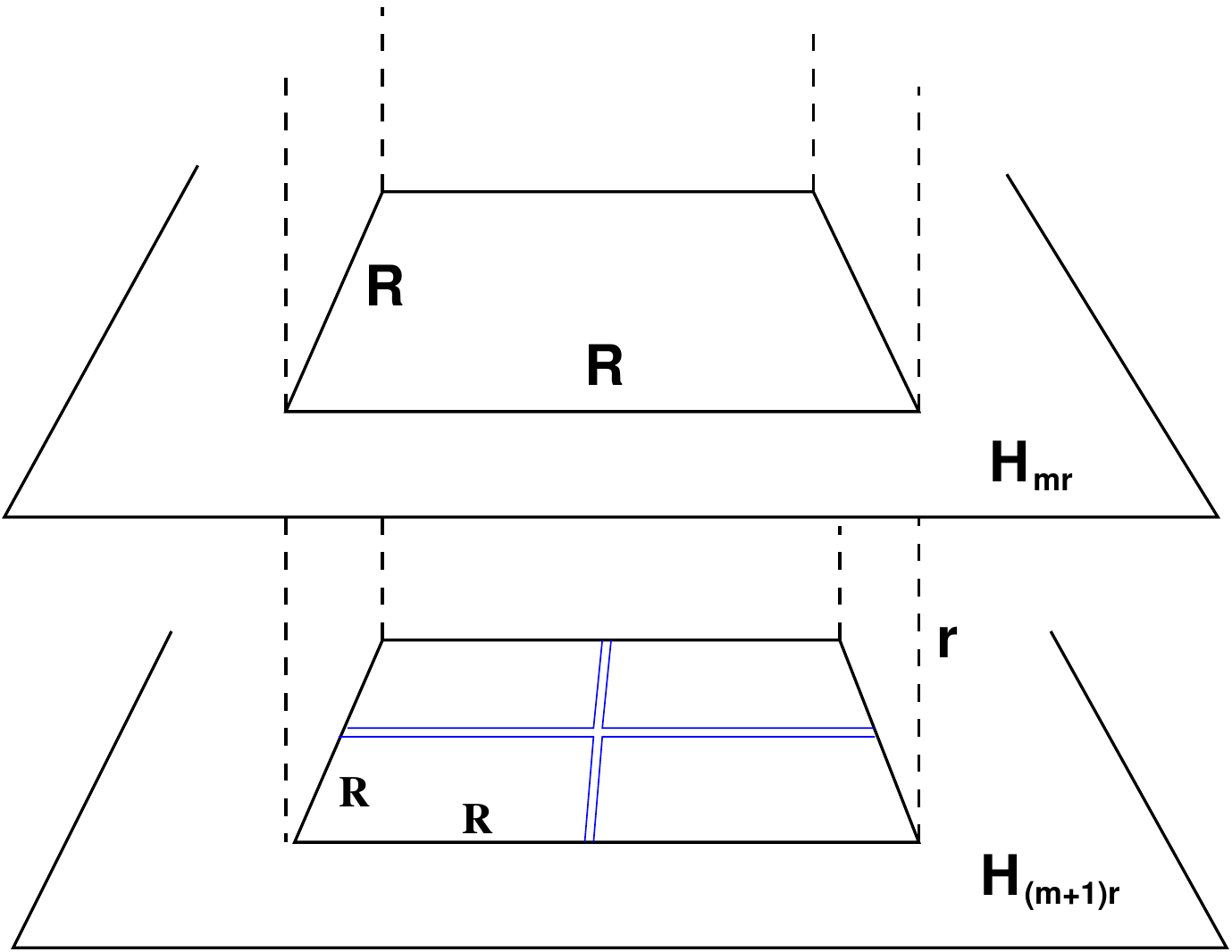} \\Figure \ref{HoroInc}: $n=3$.
\end{center}

For $m\geq 1$, let the closed disjoint cubes $Q_{i_1\ldots i_m}$ in $H_{m r_0}$  be already defined.
Fix a \emph{cube} $Q_{i_1\ldots i_m}$, then, as above, the shadow
\be \nonumber
	\mathcal{S}(Q_{i_1\ldots i_m}; H_{(m+1)r_0}) \subset  H_{(m+1)r_0}
\ee
contains $2^{n-1}$ disjoint isometric copies $Q_{i_1\ldots i_m j}$ of $Q$, $j\in \{1,\ldots ,2^{n-1}\}$.
Hence, for an alphabet $\mathcal{A}=\{1,\ldots ,2^{n-1}\}$, we associate a finite word $[w(1)\ldots w(m+1)] \in \Sigma^+(m+1)$ to the cube $Q_{i_1\ldots i_{m+1}}$ in $H_{(m+1)r_0}$ where $w(n)=i_n$ for all $n \in \{1,\ldots ,m+1\}$. 
In particular, we obtain a bijection of finite words $\Sigma^+(m)$ of length $m$ with the set of cubes 
\be \nonumber
	\mathcal{Q}(m) := \{Q_{i_1\ldots i_{m}} \subset H_{m r_0} : i_n \in \{1,\ldots ,2^{n-1}\} \text{ for } 1\leq n \leq m\}.
\ee
We denote the closed cubes $Q_{i_1\ldots i_{m}}$ obtained in this way by $q(1)\ldots q(m)$ where $q(n) \in \{1,\ldots ,2^{n-1}\}$ for $n \in \{1,\ldots ,m\}$.
Every sequence of cubes 
$\{ q(1)q(2)\ldots q(m) \}_{m \in \N}$, successively shadowed from the previous ones, determines a unique point 
\be\nonumber
	\eta := \bigcap_{m\in \N} \cal{S}(q(1)\ldots q(m);  \R^{n-1}) \in \R^{n-1},
\ee
since $\cal{S}(q(1)\ldots q(m); \R^{n-1})$, $m\in \N$, is a sequence of closed nested subsets of $\R^{n-1}$ with  diameters converging to $0$.
Define $\eta =:q(1)q(2)\ldots $ in $\R^{n-1}$.
By construction, the  geodesic line  $[\infty,\eta]$ runs through every cube $q(1)\ldots q(m)$, $m\in \N$, of the particular sequence. 
Hence, we obtain a bijection of infinite sequences $q(1)q(2)\ldots $ of cubes  
and words $w= :[w(1)w(2)\ldots ]$ in $\Sigma^+$. 
\\

\emph{Notation.}
Given a cube $q(1)\ldots q(m)$ in $\cal{Q}(m)$ and an integer $n\leq m$, let $q(1)\ldots q(m)\lvert_{n} \in \cal{Q}(n)$ be the unique cube such that $q(1)\ldots q(m)$ lies in the shadow of $q(1)\ldots q(m)\lvert_n$. Moreover, for $\xi \in \R^n$ we denote the geodesic subsegment 
$[i,j](\xi)$ by
\be \nonumber
	[i,j](\xi) := [\infty, \xi] \lvert_{[i r_0, j r_0]} : [i r_0, j r_0] \to \H^n,
\ee
where we assume that $[\infty, \xi](0) \in H_0$ and that  $i, j \in \N_0$ with $ i \leq j $, which connects the horospheres $H_{ir_0}$ to $H_{jr_0}$ and is orthogonal to both. 
If $i=j$, then we write $[i](\xi) := [i,i](\xi)$ which is the orthogonal projection of $\xi$ on the horosphere $H_{ir_0}$.

We again define the \emph{admissible set} 
\be \nonumber
	A(m) := \{ (i,s)\in \N\times \N : i+s+ \bar \ell(s) =m, s>\bar s_0\},
\ee 
if $m\geq m_0 := 2+\bar s_0+\bar \ell(\bar s_0+1)$ and set $A(m)$ to be empty for $m<m_0$.

\begin{definition}
Let $\psi\in \Gamma$ be an isometry and let $i$, $s\in \N$, $l\in \N_0$. 
If $\xi \in \R^{n-1}$ such that $d(\psi ([i](\xi)), [i+s](\xi))< \bar \e_0$ and also $d(\psi ([i+l](\xi)), [i+s+l](\xi))< \bar \e_0$ we write
\be \nonumber
	\psi \big( [i, i+l] (\xi) \big)  \sim_{\bar \e_0}  [i+s, i+s+l](\xi).
\ee
\end{definition}

\noindent In particular, by convexity of the distance function, we have for all $ j\in \{0,\ldots ,l\}$,
\be \label{Convexity}
	d( \psi \big( [i, i+j] (\xi) \big),   [i+s, i+s+j](\xi))<\bar \varepsilon_0.
\ee

We are now able to translate the proof of Theorem \ref{TheoremSequence} for the existence of $\varphi$-aperiodic words into the existence of $\varphi$-aperiodic geodesics by counting good cubes: 

\begin{definition}	\label{Good}
Let $m \in \N$. A cube $q(1)\ldots q(m)$ in $\mathcal{Q}(m)$ is called \emph{good} 
if for every $\xi \in \cal{S}(q(1)\ldots q(m);  \R^{n-1}) $, every $\psi \in \Gamma$ and every $i \in \N$, $l \in \N_0$, whenever
\be \label{ConditionGeod}
 	 \psi \big( [i, i+l] (\xi) \big)  \sim_{\bar \e_0}  [i+s, i+s+l](\xi)
\ee
for some shift $s > \bar s_0$ such that $i+s+l \leq m$, then $s > \bar \varphi(l)$. Otherwise $q(1)\ldots q(m)$ is called \emph{bad}.
\end{definition}

\noindent  If the cube $q(1)\ldots q(m)$ is good, then, since $\bar \e_0<i_M$, for every $x \in q(1)\ldots q(m)$ the projection of the geodesic segment $[\infty,x]\lvert_{[r_0, mr_0]}$  into $M$
is $\bar \varphi$-aperiodic, up to length $mr_0$, with respect to  condition \eqref{ConditionVarphi} (see the proof Lemma \ref{GoodEquiv} $(2)$).

Analogously to the proof of Theorem \ref{TheoremSequence}, for $(i,s)\in A(m)$ and $m\geq m_0$, define
\be \nonumber
\begin{array}{r}
	C_{is}:=\{ q(1)\ldots q(m) \in \mathcal{Q}(m) :  \text{ for all } \xi \in \cal{S}(q(1)\ldots q(m);  \R^{n-1}) \text{ and } \psi \in \Gamma, \\[1.5mm]
	 \psi \big( [i,i+\bar \ell (s)](\xi) \big)  \not \sim_{\bar \e_0}  [i+s, m](\xi)\}
\end{array}
\ee
and let $\mathcal{C}_m$ be the set of all $C_{ij}$ for $(i,j)\in A(m)$. Note that $\mathcal{C}_m$  is empty if $m<m_0$.

With respect to these definitions, the relationship between Definitions \ref{DiscreteDef} and \ref{Good} respectively and the sets $C_{is}$ is given by the following Lemma:

\begin{lemma} \label{GoodEquiv}
$(1)$ For $m<m_0$ every cube  $q(1)\ldots q(m) \in \cal{Q}(m)$ is good.
For $m\geq m_0$, the cube $q(1)\ldots q(m)\in \cal{Q}(m)$ is good if  $ q(1)\ldots q(m)\lvert_n \in C_{is}$ for all $n\leq m$ and $(i,s)\in A(n)$.

$(2)$ Let $q(1)q(2)\ldots $ be an infinite sequence of cubes and let $\eta \in \R^{n-1}$ be the unique corresponding limit point.
The discrete geodesic $\overline{ \pi \circ [r_0, \infty)(\eta)}$ in $M$ is $\bar \varphi$-aperiodic at every time $i\in \N$ if  for all $m\in \N$ and $(i, s) \in A(m)$ the cube $q(1)\ldots q(m)$ in $\mathcal{Q}(m)$ of the sequence $q(1)q(2)\ldots $ belongs to $C_{is}$.
\end{lemma}

 \begin{proof}
For $(1)$, let first $m<m_0$. Let $i$, $s \in \N$, $l\in \N_0$ such that $s > \bar s_0$ and $i+s+l\leq m < 2+\bar s_0+\bar \ell(\bar s_0+1)$. In particular, $l< \bar \ell(\bar s_0+1)$ so that $\varphi(l)<\bar s_0+1 \leq s$ and every cube $q(1)\ldots q(m)$ follows to be good.  
 
Now let $m\geq m_0$. 
Assume by absurd that $q(1)\ldots q(m)$ is not good and let $\xi \in \cal{S}(q(1)\ldots q(m);  \R^{n-1})$ and $\psi \in \Gamma$ such that for some $i\in\N$, $l\in \N_0$, we have
\be \nonumber
	\psi \big( [i, i+l](\xi) \big) \sim_{\bar \e_0}  [i+s, i+s+l](\xi),
\ee
where $s> \bar s_0$ with $i+s+l\leq m$ and $s\leq \bar \varphi(l)$. 
Hence, $\bar \ell (s) \leq l$ and for $n:=i+s+\bar \ell(s)$ we have in particular by \eqref{Convexity}, 
\be \nonumber
	\psi \big( [i, i+\bar \ell(s)](\xi) \big) \sim_{\bar \e_0}  [i+s, n](\xi).
\ee
Hence, we see that $q(1)\ldots q(m)\lvert_n \not \in C_{is}$ where $(i,s)\in A(n)$ for $n\leq m$; a contradiction.

For $(2)$, assume that $\bar \gamma:= \overline{\pi \circ [r_0, \infty)(\eta)}$ is not $\bar \varphi$-aperiodic at time $i\in \N$. Then there must be a shift $s\in \N$ with $s > \bar s_0$,  and $l \in \N_0$ 
such that 
\be \nonumber
	d( \bar \gamma(i + j),\bar  \gamma(i+s+j)) < \bar \e_0 \ \ \ \text{ for all } j \in \{0,\ldots ,l\},
\ee
where $s \leq \bar \varphi(l)$. 
Since $\bar \e_0<i_M$ and the distance function is convex, we also have $d(\gamma((i + t)r_0),  \gamma((i+s+t)r_0) < \bar \e_0$ for all $0 \leq t \leq l$ for the corresponding extended geodesic $\gamma : \R \to M$. 
By discreteness of $\Gamma$, there exist finitely many isometries $\psi_1$,\ldots ,$\psi_q \in \Gamma$ and a subdivision of the interval $[ir_0, (i+l)r_0]$ into $[l_0r_0, l_1r_0], [l_1r_0, l_2r_0], \ldots , [l_{q-1}r_0, l_qr_0]$ where $l_0=i$ and $l_q=i+l$ and $l_j\in \R$, such that (with analogous notation as above)
\be \nonumber
	\psi_{j+1} \big( [l_j, l_{j+1}](\eta) \big) \sim_{\bar \e_0} [s+l_j, s+l_{j+1}](\eta), \ \ \  j=0,\ldots ,q-1.
\ee
We thus have $d( \psi_{j+1} \big( [l_{j+1}](\eta) \big), [s+l_{j+1}](\eta)) < \bar \e_0$ and $d( \psi_{j+2} \big( [l_{j+1}](\eta) \big), [s+l_{j+1}](\eta)) < \bar \e_0$.
Since $\bar \e_0 < i_M$ and every orbit of $\Gamma$ is $2i_M$-separated (that is, for $\psi$, $\bar \psi \in \Gamma$ we have $d(\psi x, \bar \psi x) \geq 2 i_M$ for any $x\in \H^n$) it follows from the triangle inequality that $\psi_{j+1} \big( [l_{j+1}](\eta) \big) =  \psi_{j+2} \big( [l_{j+1}](\eta) \big)$; hence $\psi_{j+1} = \psi_{j+2}$ for all $j =0,\ldots ,q-2$ since $\Gamma$ acts freely.
Therefore, we have an isometry $\psi \in \Gamma$ such that  
 \be \nonumber
 	 \psi \big( [i, i+l](\eta) \big) \sim_{\bar \e_0}  [i+s, i+s+l](\eta)
\ee
 where $s\leq \bar \varphi(l) $. The proof is now finished analogously to the case of $(1)$.
\end{proof}

\noindent In view of Lemma \ref{GoodEquiv}, let for $m \geq m_0$,
\be \nonumber
	\cal{Q}^g(m) = \{ q(1)\ldots q(m) \in \mathcal{Q}(m) : q(1)\ldots q(m)\lvert_n \in C_{is} \text{ for all } (i, s)\in A(n), n \leq m\},
\ee
 and   $Q^g(m)= \mathcal{Q}(m)$ for $m< m_0$, which is a subset of all good cubes at step $m$. 

\begin{lemma}  Assume that condition \eqref{D} is satisfied.
Then, for $m\in \N$,
\be \label{Increasement}
	\lvert \cal{Q}^g(m+1) \rvert \geq k \lvert \cal{Q}^g(m) \rvert -  \bar c \cdot \sum_{ C_{is} \in \mathcal{C}_{m+1}}  \lvert \cal{Q}^g(i+s-1) \rvert,
\ee
where $\bar c$ is a constant depending only on $n$, $i_M$ and $\bar  s_0$, and is strictly decreasing in $\bar s_0$. 
\end{lemma}

\begin{proof} 
If $m+1 < m_0$ then $\mathcal{C}_{m+1}$ is empty and the claim follows. 
Hence assume $m+1 \geq m_0$.
Let 
\be\nonumber
	L= \{ q(1)\ldots q(m+1) \in \mathcal{Q}(m+1) : q(1)\ldots q(m+1)\lvert_m \in \cal{Q}^g(m) \}
\ee
 and note that $\lvert L \rvert = k \lvert Q^g(m) \rvert$. Then
\be\nonumber
	\cal{Q}^g(m+1) = L \cap (\bigcap_{C_{is} \in \mathcal{C}_{m+1}} C_{is}) 
	= L \setminus ( \bigcup_{C_{is} \in \mathcal{C}_{m+1}} (L \cap C_{is}^C)),
\ee
where $C_{is}^C$ is the complement of $C_{is}$. 
Fix some $C=C_{is} \in \mathcal{C}_{m+1}$.
Define 
\be	\nonumber
	Q =\{ q(1)\ldots q(m+1) \lvert_{i+s-1} \in \mathcal{Q}(i+s-1) : q(1)\ldots q(m+1) \in L \},
\ee
One checks that $\lvert Q \rvert \leq \lvert \cal{Q}^g(i+s-1) \rvert$. Let $L = \cup_{q \in Q} L_q$ where 
 \be \nonumber
 	L_q = \{q(1)\ldots q(m+1) \in L : q(1)\ldots q(m)\lvert_{i+s-1} =q\}.
\ee
It remains to show that each $L_q \cap C^C$ contains at most $\bar c $ cubes; in this case,
\be\nonumber
	\lvert L \cap C^C \rvert  \leq 
	\bar c \cdot \lvert Q\rvert \leq
	\bar c \cdot \lvert \cal{Q}^g(i+s-1) \rvert.
\ee
The following claim concludes the proof.
\end{proof}

\begin{claim} \label{Claim}
 $\lvert L_q \cap C^C \rvert \leq \bar c \cdot \lvert \cal{Q}^g(i+s-1) \rvert$.
\end{claim}

\noindent For the proof of the claim note that if  \eqref{C} is satisfied, then for all $l \geq \bar s_0$,
\be \nonumber
   \lfloor \bar \varphi(l) \rfloor > l,
\ee
which implies that  for all $s > \bar s_0$, 
\be \label{BedingungL0}
	\bar \ell(s)<s.
\ee
To see this, assume $\bar \ell(s) \geq s$ for some $s >\bar  s_0$. 
Then, by definition of $\bar \ell$, $\bar \varphi(j) <s$ for all $s>j \in \N_0$.
In particular, for $\bar s_0 < s $ we have $\bar \varphi(\bar s_0) \geq \lfloor \bar \varphi(\bar s_0) \rfloor$; a contradiction to $ \lfloor \bar \varphi(\bar s_0) \rfloor > \bar s_0$.

\begin{proof}[Proof of the Claim \ref{Claim}.]
$L_q $ consists of cubes of the form $q \cdot q(i+s)\ldots q(m+1) \in \cal{Q}(m+1)$.
Hence, consider the point set $W$ of all geodesic segments $[i, i+\bar \ell (s)](\xi)$ where $\xi \in \cal{S}(q, \R^{n-1})$; 
see Figure \ref{W}.
Since $s>\bar s_0$ we have $\bar \ell(s)<s$ by \eqref{BedingungL0}, and therefore $s-1-\bar \ell(s) \geq 0$.
Moreover, by definition, the cube $q$ in $H_{(i+s-1)r_0}$ has $h$-edge lengths $R$.
Thus from \eqref{Heintze}, the subset $H_{i+\bar \ell (s)} \cap W$ is isometric to an Euclidean cube with $h$-edge length
\be \nonumber
	 e^{-(i+s-1)r_0 + (i+\bar \ell(s))r_0} R= e^{-(s-1 - \bar \ell(s))r_0} R \leq R.
\ee
Since an Euclidean cube in $\E^{n-1}$ of edge length $L$ has diameter at most $\sqrt{n-1}L$, we obtain from \eqref{Heintze2} that the $d$-diameter of $H_{i+\bar \ell (s)} \cap W$ is bounded above by
\be \label{Length1}
	2 \arcsinh(e^{-(s-1 - \bar \ell(s))r_0} \sqrt{n-1} R/2).
\ee
In the same way, the $h$-edge length of $H_{i r_0} \cap W$ is given by 
\be \label{Bound}
	e^{-(s-1)r_0} R.
\ee

Now, by definition, for every $q \cdot q(i+s)\ldots q(m+1) \in L_q \cap C^C$ there exists
$\psi \in \Gamma$ such that  $\psi \big( [i,i+\bar \ell (s)](\xi) \big) \sim_{\bar \e_0} [i+s, m+1](\xi) $ for some  $\xi \in \cal{S}(q, \R^{n-1})$. 
In particular, $x:=[m+1](\xi)$ must belong to the $\bar \e_0$-neighborhood of $\psi(W \cap H_{i+s+\bar \ell(s)})$.
Thus, we want to estimate the maximal number of cubes in $\cal{Q}(m+1)$ which  intersect with the $\bar \e_0$-neighborhood of $\psi(W \cap H_{i+s+\bar \ell(s)})$.
Let therefore also $y \in H_{(m+1)r_0}$ belong to the $\bar \e_0$-neighborhood of $\psi(W \cap H_{i+s+\bar \ell(s)})$.
By the triangle inequality and by \eqref{Length1}, we have
\be	\nonumber
	d(x,y) \leq 2 \bar \e_0 + 2 \arcsinh(e^{-(s-1 - \bar \ell(s))r_0} \sqrt{n-1} R/2).
\ee
Therefore, again from  \eqref{Heintze2}, the $h$-diameter of the intersection of the $\bar \e_0$-neighborhood of $\psi(W \cap H_{i+s+\bar \ell(s)})$ with $H_{(m+1)r_0}$ is bounded above by 
\be \nonumber
	 \bar r_1(s) := 2\sinh( \bar \e_0 +  \arcsinh(e^{-(s-1 - \bar \ell(s))r_0} \sqrt{n-1} R/2) ).
\ee
On the other hand, the cubes $q\cdot q(i+s)\ldots q(m+1) \in \cal{Q}(m+1)$ are disjoint and have Euclidean volume $R^{n-1}$.
Therefore, we set
\be \nonumber
 	\bar c_1(s):= \lceil \frac{ (\bar r_1(s) + \sqrt{n-1}R)^{n-1} }{ R^{n-1}  } \rceil.  
\ee
Hence,  the $\bar \e_0$-neighborhood of $\psi(W \cap H_{i+s+\bar \ell(s)})$ can intersect at most $\bar c_1 (s)$ qubes in $\mathcal{Q}(m+1)$.
Since $q(1)\ldots q(m)$ is good for every $q(1)\ldots q(m+1) \in L_q$, 
we conclude that, with respect to $\psi$, at most $\bar c_1(s)$ cubes can become bad in $L_q \cap C^C$.

\begin{center} \label{W}
 \includegraphics[scale=0.4]{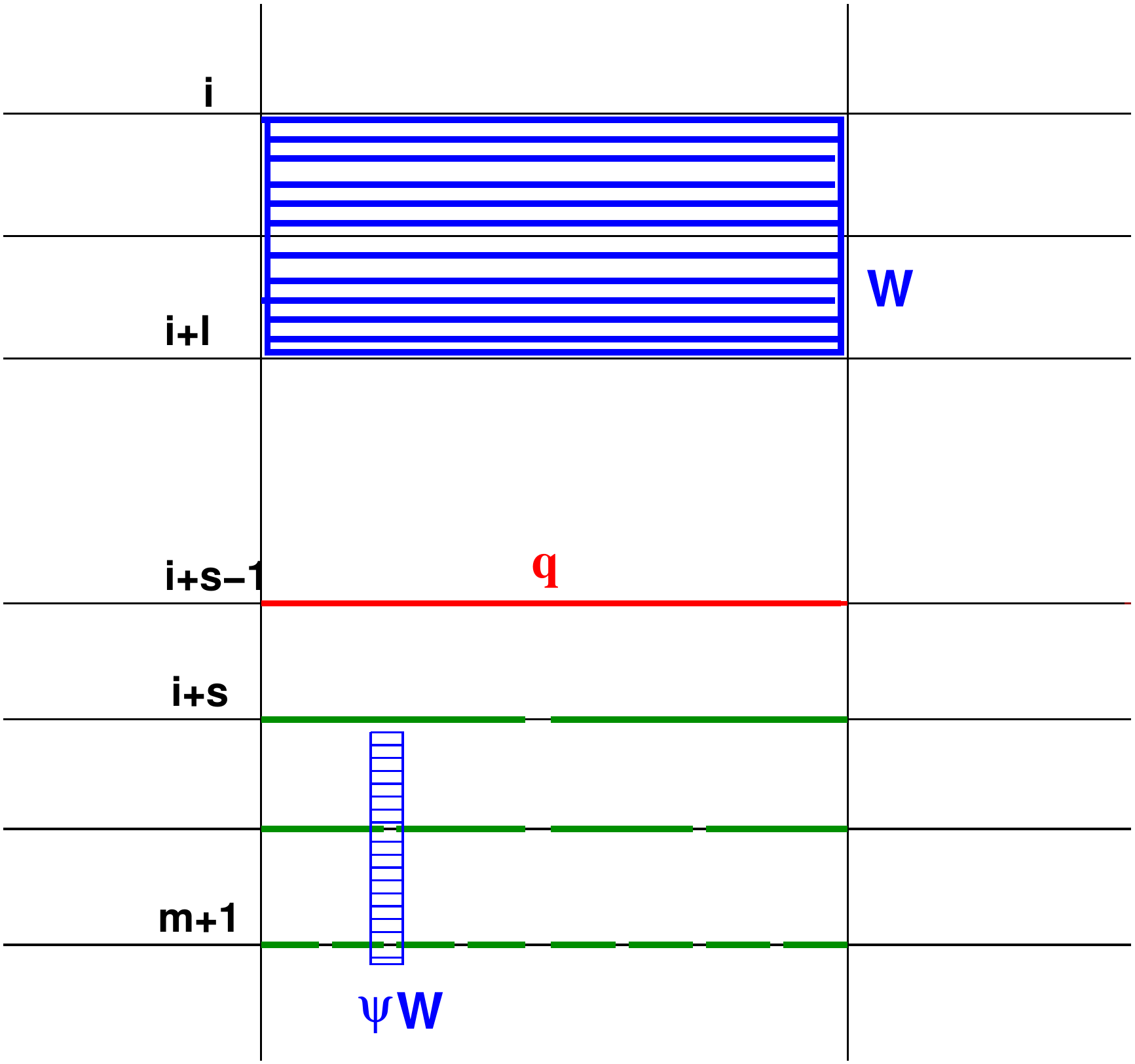} \\Figure \ref{W}: $n=2$.
\end{center}

Now, let $\bar y$ be the center of $W \cap H_{i r_0}$, which is isometric to a cube in the Euclidean space  of edge length $e^{-(s-1)r_0}R$ by \eqref{Bound} and contained in the cube $q \lvert_{i}$. From \eqref{Heintze2}, $W \cap H_{i r_0}$  must be contained in the hyperbolic ball $B_d(\bar y, \bar r_2(s))$, where 
\be \nonumber
	\bar r_2(s)=2 \arcsinh(e^{-(s-1)r_0}\sqrt{n-1}R/4).
\ee
Note that if there is some point $p \in W \cap H_{i r_0}$ and some $\psi \in \Gamma$ such that $d(\psi p, \bar q) < \bar \e_0$, where $\bar q := \cal{S}(q, H_{(i+s)r_0})$, then $d( \psi \bar y, \bar q) < \bar \e_0+ \bar r_2(s)$.
In particular, for every cube $q \cdot q(i+s)\ldots q(m+1) \in L_q \cap C^C$ there exists such an isometry $\psi$. 
But since the orbit $\Gamma \bar y$ is $2 i_M$-separated, the open metric balls $B(\psi \bar y, i_M)$, $\psi \in \Gamma$, are disjoint and there can only be finitely many, say $\bar c_2(j)$, intersecting the max$\{\bar \e_0+\bar r_2(s)-i_M,0\}$-neighborhood of $\bar q$. 
In fact, from \eqref{Heintze} and \eqref{Heintze2}, the $h$-diameter of $\bar q$ is bounded above by $e^{r_0}\sqrt{n-1}R$ and $\bar q$ must be contained in a hyperbolic ball of radius  $ 2 \arcsinh(e^{r_0}\sqrt{n-1}R/4)$. 
Therefore, $\bar c_2(s)$ is bounded above by
\be \nonumber
	 \lceil \frac{ \text{vol}(B \big( 2 \arcsinh(e^{r_0}\sqrt{n-1}R/4) +  2\arcsinh(e^{-(s-1)r_0}\sqrt{n-1}R/4) + \bar \e_0 \big) )} { \text{vol}(B(i_M/2)) } \rceil.
\ee

Since both, $\bar c_1(s)$ and $\bar c_2(s)$ are non-increasing in $s$, we conclude the claim by setting $\bar c:= \bar c_1 (\bar s_0+1)\bar c_2(\bar s_0+1)$.
\end{proof}

Analogously to the proof of Lemma \ref{CountConditions}, the previous Lemma yields the following.

\begin{lemma} \label{Increasement2}
Assume that condition \eqref{BedingungL0} is satisfied.
Then, for $m\in \N$,
\bea \nonumber
 \lvert \mathcal{Q}^g(m+1)\rvert&\geq &\big( k 
  -  \mathbf{1}_{\{\bar \ell(\bar s_0+1)=0\}} \bar c  \lfloor \bar \varphi(0) \rfloor \big) \lvert \mathcal{Q}^g(m) \rvert   \nonumber\\
  &-&\textstyle{ \bar c  \cdot \sum_{j=\max(\bar \ell(\bar s_0+1),1)}^{m}  (\lfloor \bar \varphi(j) \rfloor -\lfloor \bar \varphi(j-1) \rfloor )  \lvert \mathcal{Q}^g(m-j) \rvert}.\nonumber
\eea
\end{lemma}

\begin{proof}
Recall the definition of the set $H_j= \{C_{is} \in \cal{C}_{m+1}: i+s-1 = m-j\}$ in \eqref{Hs}. 
Since $\bar \ell$ is non-decreasing we have $j= m+1 -(i+s) = \bar \ell(s) \geq \bar \ell(\bar s_0+1)$ if $s>\bar s_0$.
\end{proof}

Finally, if moreover condition \eqref{C} is satisfied, then the same inductive proof as in Lemma \ref{AsymptoticGrowth} shows that the number of good cubes in $Q^g(m+1)$  increases in $m+1$ by the factor $c>1$; see \eqref{Formula}. 
Lemma \ref{GoodEquiv}.$(2)$ then shows the existence of a $\bar \varphi$-aperiodic geodesic $\bar \gamma : \N \to M$. 
Thus, we have shown the following.

\begin{lemma} Assume that conditions \eqref{D} and \eqref{C} are satisfied.
Then, for $m\in \N$,  $\lvert \mathcal{Q}^g(m) \rvert \geq c^m$.
In particular, there exists  a $\bar \varphi$-aperiodic geodesic $\bar \gamma : \N \to M$ with parameters $(\bar s_0, \bar \e_0, r_0)$.
\end{lemma}

Now, let $\bar \gamma : \N \to M$ be a $\bar \varphi$-aperiodic geodesic  (with parameters $(\bar s_0, \bar \e_0, r_0)$ and let $\gamma: \R \to M$ be the corresponding extended geodesic. Consider the sequence $v^n:=\phi^n \gamma'(r_0)$, $n\in \N$, in the compact space $SM$ and let $\gamma_0$ be an accumulation point.
The space of unit speed geodesics (identified with $SM$) is endowed with the topology of uniform convergence on bounded sets. Therefore note that a sequence $v^n$ converges to $v$ in  $SM$ if and only if for every $l\geq0$ and every $\tau>0$ there exists $N\in \N$ such that  for every $n\geq N$, $d(\gamma_{v^n}(t), \gamma_v(t)) < \tau$ for every $t\in [-l,l]$. Therefore $\bar \varphi$-aperiodicity can be shown to be a closed condition (similarly as in  Lemma \ref{ClosedCond}).
Since $\bar \gamma_{v^n}$ is $\bar \varphi$-aperiodic  beginning at $t_n \geq -(n-1)$ (with parameters $(\bar s_0, \bar \e_0, r_0)$), it follows that $\bar \gamma_0 : \Z \to M$ is $\bar \varphi$-aperiodic.
This completes the proof of Theorem \ref{TheoremGeodesic}.

\subsection{Proof of Theorem \ref{MainThm}.} 
For $\delta\in (0,1)$ choose $\bar \delta\in [\delta,1)$ such that for $r_0= \ln (3-\bar \delta)$ we have $\ln (3-\bar\delta) + \e_0 < i_M$.
Note  that $\tilde\delta = \bar \delta \ln(2) / \ln(3-\bar \delta)\to 1$ as $\bar \delta\to 1$ and assume therefore that $\tilde\delta>\delta$.
For $l\geq 0$ let $\bar \psi(l) = 2^{\bar \delta(n-1)l}$  so that its right inverse  $\lceil \frac{1}{\bar \delta(n-1)\ln(2)} \ln(s) \rceil$ is an unbounded function.
Then, for $c=\frac{1}{2}(2^{n-1}+2^{\bar \delta(n-1)})$, we have that for sufficiently large $\bar s_0 = \bar s_0(\bar \delta, n,i_M, \e_0) \in \N_0$ the conditions \eqref{D} and \eqref{C} are satisfied.
Thus, from Theorem \ref{TheoremGeodesic} there exists a discrete geodesic $\bar \gamma: \Z \to M$ which is $\bar\psi$-aperiodic with respect to $(\bar s_0, r_0+\varepsilon_0, r_0)$.
From Lemma \ref{ContCond} we obtain that $\gamma : \R \to M$ is continuously $ \psi$-aperiodic with parameters $s_0=(\bar s_0+1)r_0$ and $\e_0$, where for $l\geq r_0$,
\be\nonumber
\begin{array}{lcl}

		       \psi( l) &=& \ln(3- \bar \delta) \cdot \bar \psi(\frac{l}{\ln(3- \bar \delta)}-1)  - \ln(3- \bar \delta) \\
				&=& \frac{ \ln(3- \bar \delta)}{2^{\bar \delta(n-1)}} e^{\frac{\bar \delta\ln(2)}{ \ln(3- \bar \delta)} (n-1) l} -  \ln(3- \bar \delta) \\
				
				&=& \big( \frac{ \ln(3- \bar \delta)}{2^{\bar \delta(n-1)}} - \frac{\ln(3- \bar \delta)}{e^{  \tilde\delta (n-1)l }} \big) e^{\tilde\delta(n-1)l}  \\
				&=:& c(  \tilde\delta,l) \cdot e^{  \tilde\delta(n-1)l }= c( \tilde\delta,l) \varphi_{ \tilde\delta}(l).
\end{array}
\ee
Note that $c(  \tilde\delta, l)$ is increasing in $l$ and we restrict $\psi$ to the interval $[l_1,\infty)$ for some $l_1>\ln(3- \bar \delta)$ such that $c( \tilde\delta, l_1)>0$. 

We now translate the minimal shift $s_0$ into the minimal length $l_0$.  Let to this end $N:= \lceil \frac{ s_0}{2i_M}\rceil$. Assume that for some $t_0$ we have $d( \gamma(t_0 + t) ,  \gamma(t_0 +s +t) < \varepsilon_0$ for all  $0\leq t \leq l$ where $l\geq \max\{l_1, 3N  s_0 + 2i_M\} =:l_0$. 

First, we assume that $s \leq s_0$. 
Note that the function $t\mapsto d( \gamma(t_0 + t) , \gamma(t_0 +s +t)$ is not only convex but decreases and increases exponentially (see \cite{Brid}) so that we have 
$d(\gamma(t_0 + t) ,  \gamma(t_0 +s +t)< \e_0/4$ for all $s' \leq t \leq l-s'$ where $s'$ is sufficiently large, say $s'=2i_M$.
The closing lemma implies the existence of a closed geodesic nearby; in fact, we will prove the following Lemma.

\begin{lemma} \label{ClosedGeodesic} In this setting, there exists a closed geodesic $\alpha$ of period $p\leq s+\e_0/4$ such that (up to parametrization of $\alpha$),
\be \nonumber
	d(\alpha(t), \gamma(t_0+s' +t)) < \e_0/2  \quad \text{ for all } 0\leq t \leq s+ l-2s' - \e_0.
\ee
\end{lemma}

\noindent Let $N' = \lceil  s_0 / p \rceil \in \N$ be the smallest integer such that $ N'p>\bar s_0$ and note that $2Ns \geq N'p$. We then have by the triangle inequality,
\bea \nonumber
&&d(\gamma(t_0+s' + t), \gamma(t_0 +s'+ N'p + t)) \\
&\leq& d(\gamma(t_0+s' +t), \alpha(t)) +d(\gamma(t_0+s'+N'p +t), \alpha(t))< \e_0\nonumber
\eea
for all  $0\leq t \leq l-2s'-N'p+s$ and in particular for all $0\leq t \leq l-2s'-2Ns_0$.
Thus, 
\be \nonumber
	2Ns  \geq N'p > c( \tilde\delta,l_1) \varphi_{\tilde\delta}(l-2s'-2N s_0)) = \frac{c(\tilde\delta, l_1)}{ e^{\tilde\delta(n-1)(2s'+2N\bar s_0)}} \varphi_{ \tilde\delta}(l),
\ee
and we can find a positive constant $c_0=c_0(\tilde\delta,i_M,n,\e_0)$ such that $s>c_0\varphi_{ \tilde\delta}(l)$.

In the case when $s> s_0$, we have
\be \nonumber
s> c( \tilde\delta,l_1) \varphi_{ \tilde\delta}(l)  \geq \ c_0 \varphi_{ \tilde\delta}(l).
\ee

Finally, since $\delta< \tilde\delta$, we restrict if necessary to $\tilde l_0 \geq l_0$ such that $c_0 \varphi_{ \tilde\delta}(l) \geq \varphi_{ \delta} (l)$ for all $l\geq \tilde l_0$.
The proof of Theorem \ref{MainThm} is finished by the proof of Lemma \ref{ClosedGeodesic}.

\begin{proof}[Proof of Lemma \ref{ClosedGeodesic}]
We consider the setting of the proof of Theorem \ref{TheoremGeodesic}.
Let now $d_M$ be the distance function on $M$ and recall that we have $d_M(\gamma(t_0 + t) , \gamma(t_0 +s +t)< \e_0/4$ for all $s' \leq t \leq l-s'$, where $s'=2i_M > 2\ln(2)$.
We denote a lift of the segment $\gamma$ on $[t_0+s', t_0+l-s']$ by $\beta$ and let the endpoints of $\beta$ be  $x_1$ and $x_2$. Since $\e_0<i_M$, there exists an isometry $\psi \in \Gamma$ such that $d(\beta, \psi (\beta(t)))< \e_0/4$ for all $t \in [t_0+s', t_0+l-s']$ and in particular, $d(x_i, \psi x_i) < \e_0/4$ for $i=1,2$.
Let $\tilde \alpha$ be the axis of $\psi$ and denote by $d_1=d(\tilde \alpha, x_1)$ and $d_2=d(\tilde \alpha,x_2)$.
We first show that $d_1$ is close to $d_2$ in the following sense.
Namely, the displacement function $d_{\psi}(\cdot)=d(\psi \cdot, \cdot)$ grows at least linearly in the distance to $\tilde \alpha$.
Since $s-\e_0/4\leq d_{\psi}(x_i)\leq s+\e_0/4$ for $i=1,2$ we see that $\lvert d_1 -d_2\rvert$ is bounded by a constant depending only on $\psi$, $s$ and $\e_0$.

Now, if we show that $d_i<\e_0/2$ for $i=1,2$, then the proof follows by convexity of the distance function.
We show this for $d_1$.
Since $d_1$ is close to $d_2$ and $l$ is large, the distance function $t\mapsto d(\beta(t), \tilde \alpha(t))$ decreases exponentially on $[0,s']$, where $\tilde \alpha$ is parametrized such that $\tilde \alpha(0)$ equals the orthogonal projection $\bar x_1$ of $x_1$ on the convex set $\tilde \alpha$.
Moreover, $s'$ is large and thus $d(\tilde \alpha, \beta(s')) < d_1/2$.  The orthogonal projection of $\psi (x_1)$ on $\tilde \alpha$ is given by $\psi (\bar x_1)$. Hence, $d(\psi (x_1), \tilde \alpha(s')) \geq d(\psi (x_1), \psi ( \bar x_1))  = d_1$.
On the other hand, we have by the triangle inequality
$d(\psi (x_1), \tilde \alpha(s')) \leq d(\psi (x_1), \beta(s')) + d(\beta(s'), \tilde \alpha(s')) < d_1/2 + \e_0/4$.
Thus, $d_1 <d_1/2 + \e_0/4$ and the claim follows.
\end{proof}

\bibliographystyle         {plain}      
\bibliography{cup_ref}

 \end{document}